\documentclass[11pt]{article}
\usepackage[centertags,intlimits]{amsmath}
\usepackage{amsfonts}
\usepackage{amsthm}
\usepackage{graphics}
\usepackage{color}
\usepackage{graphicx}
\usepackage{amssymb}
\allowdisplaybreaks[2]
\numberwithin{equation}{section}
\newtheorem{theorem}{Theorem}[section]
\newtheorem{lemma}{Lemma}[section]

\newtheorem{remark}{Remark}[section]

\providecommand{\abs}[1]{\lvert #1\rvert}
\providecommand{\norm}[1]{\lVert #1\rVert}

\newcounter{neweqn}
\newcommand{\beq}[1]{\addtocounter{neweqn}{1}\begin{equation}\label{#1}}
\newcommand{\eeq}{\end{equation}}

\newcommand{\nc}{\newcommand}

\newcommand{\beray}[1]{\addtocounter{neweqn}{1}\begin{eqnarray}  \label{#1}}

\nc{\vb}{\mathbf{v}}
\nc{\bx}{\mathbf{x}}
\nc{\by}{\mathbf{y}}
\nc{\bz}{\mathbf{z}}
\nc{\bu}{\mathbf{u}}
\nc{\bv}{\mathbf{v}}
\nc{\ba}{\mathbf{a}}
\nc{\bs}{\mathbf{s}}
\nc{\bq}{\mathbf{q}}
\nc{\bd}{\mathbf{d}}
\nc{\bb}{\mathbf{b}}
\nc{\bc}{\mathbf{c}}
\nc{\bi}{\mathbf{i}}
\nc{\bfr}{\mathbf{r}}
\nc{\bP}{\mathbf{P}}
\nc{\bQ}{\mathbf{Q}}
\nc{\R}{\mathbb R}
\nc{\N}{\mathbb N}
\nc{\bbC}{\mathbb C}
\nc{\D}{\mathbb D}
\nc{\Z}{\mathbb Z}
\nc{\F}{\mathbf F}
\nc{\bbS}{\mathbb S}
\nc{\bE}{\mathbf E}
\nc{\B}{\cal B}
\nc{\br}{\bigr}
\nc{\bl}{\bigl}
\nc{\Bl}{\Bigl}
\nc{\Br}{\Bigr}
\nc{\ind}[1]{\,\mathbf{1}_{\{#1\}}\,}

\author{Anatolii A. Puhalskii
}

\title{Moderate deviations of many--server queues, \\
 idempotent processes and  quasipotentials 
}

\begin{document}
\maketitle
\sloppy

\begin{abstract}
A Large Deviation Principle (LDP)  is
established for the stationary distribution of
the number of customers in a many--server queue in
heavy traffic for a  moderate deviation scaling akin to the
Halfin--Whitt regime. The interarrival and service times are assumed
generally distributed. The deviation function is given by a
quasipotential. It is  related to the longterm
idempotent distribution of the large deviation limit of the
number-in-the-system process.  New results on the 
 trajectorial LDP for that process are also obtained.
Proofs rely on the characterisation of large deviation relatively
compact sequences as  exponentially tight ones and use methods of weak
convergence and idempotent processes. Another contribution of the
paper concerns    bounds on the
higher--order moments of
counting renewal processes.
\end{abstract}

\section{Introduction}
\label{sec:introduction}

The analysis of many--server queues with generally distributed
interarrival and service times
 is challenging.   In order to gain
insight, various asymptotics
have been explored, see, e.g., Borovkov \cite{Bor67q,Bor80}, Iglehart  \cite{Igl65,Igl73} and
Whitt \cite{Whi82}.
 Of particular interest is the
heavy--traffic  asymptotic regime  introduced in Halfin and
Whitt \cite{HalWhi81} where
  the service time distribution is held fixed, whereas 
 the number of servers, $n$\,, and
the arrival rate, $\lambda_n$\,, 
grow without a bound in such a way that 
\[
\sqrt{n}(1-\rho_n)\to\beta\in\R\,,
\]
 where $\rho_n=\lambda_n/(n\mu)$ is the traffic intensity, $\mu$
representing the service rate.
With  $Q_n(t)$ denoting the number of customers present at
time  $t$\,, the
sequence of processes $(Q_n(t)-n)/\sqrt{n}$\,,
considered as  random elements of the associated Skorohod space,
under fairly general hypotheses, 
converges in law to a continuous--path process. 
In their  pioneering contribution,
   Halfin and Whitt \cite{HalWhi81} assumed
 renewal arrivals and
 exponentially distributed service times. The limit process then is an Ito
 diffusion with a piecewise linear drift.
In addition, for the case where  $\beta>0$\,, 
 the stationary
 distribution of the centred and normalised number--in--the--system
 process converges
 to the stationary distribution of
 the limit, which has a continuous density with an exponential upper
 tail and a normal lower tail. The limiting probability of a  customer having 
to await service is strictly between zero and one, which 
sets  the Halfin--Whitt regime apart
  and renders it useful for producing approximations.

 An extension of the process--level convergence to the case of generally
 distributed service times  is due to Reed
 \cite{Ree09}, see also Puhalskii and Reiman \cite{PuhRei00}, 
Puhalskii and Reed \cite{PuhRee10}. The limit
 process, however,  is not easily tractable, being a solution to a
 nonlinear renewal equation.
In an attempt at ''Markovisation'',
 Kaspi and Ramanan \cite{KasRam13}  incorporated elapsed service
times into state
descriptions
 and  arrived at infinite--dimensional limits. 
 Convergence of the stationary distribution of the
 Markovised number of customers was  addressed in  Aghajani and Ramanan
 \cite{AghRam20}. It was assumed
that the service time cdf
 has a 
density with certain smoothness and
 integrability properties. 
 The limit infinite--dimensional distribution
  is quite involved.
Special cases of the convergence of the stationary distribution were treated in 
Jelenkovi{\'c}, Mandelbaum and 
Mom{\v{c}}ilovi{\'c} \cite{JelManMom04} \footnote{ Theorem 1 in Jelenkovi{\'c}, Mandelbaum and 
Mom{\v{c}}ilovi{\'c} \cite{JelManMom04}
is missing a Lindeberg condition on the interarrival times.},
 who assumed deterministic
service times
and in Gamarnik and Mom{\v{c}}ilovi{\'c} \cite{GamMom08}, who
assumed 
 lattice--valued  service times   of finite support.

This paper undertakes a study of moderate deviations of 
 the  stationary   number of customers.
It assumes  the    heavy--traffic condition
\begin{equation}\label{eq:105}
  \frac{\sqrt{n}}{b_n}(1-\rho_n)\to\beta\,,
\end{equation}
where $b_n\to\infty$ and $b_n/\sqrt{n}\to0$\,.
The main result  is a Large Deviation Principle (LDP) for the stationary distribution of
the process $X_n=(X_n(t)\,,t\in\R_+)$\,, where
$  X_n(t)=(Q_n(t)-n)/(b_n\sqrt{n})$\,, provided
$\beta>0$\,. Unlike in  Aghajani and Ramanan
 \cite{AghRam20}, no fine properties of the service time cdf are
 assumed, moment conditions being imposed instead.
It is noteworthy that  deviation functions, a.k.a. rate functions,
a.k.a. action functionals,
 for moderate deviations
may capture  exponents in
the corresponding weak convergence limits, see Puhalskii
\cite{Puh99a}.

It has been an established practice  to describe large deviation
asymptotics of the invariant measures of Markov processes
 in terms of quasipotentials, which involve
trajectorial deviation
functions,  see
Freidlin
 and Wentzell \cite{wf2}, Schwarz  and Weiss \cite{SchWei95}.
In Puhalskii
 \cite{Puh19a} it was shown that a similar recipe may work for
non--Markov processes  that
  are  Markovian ''in the large deviation limit'', the latter meaning that the
integrand in the expression for the trajectorial
 deviation function 
   depends only on
 the state
 variable and its derivative.
 For many--server queues with generally distributed service times, however,
 not only is the process $X_n$
 non--Markov but the trajectorial 
deviation function 
  is
essentially ''non--Markov'' either in that the integrand 
 depends on the entire past of the trajectory, see Puhalskii \cite{Puh23}.
Notwithstanding, 
this paper shows
 that moderate deviations of the stationary distribution of
 $X_n$ are still governed  by the associated quasipotential.

A standard approach to establishing an LDP for the invariant
measures of a
Markov process has been to use formulas that relate
the invariant measure  to the  cycle lengths of  appropriately chosen
Markov chains and verify
the bounds in the definition of the LDP,
 a change of probability measure playing a prominent
role, cf. Freidlin and Wentzell \cite{wf2}, Schwarz and Weiss
\cite{SchWei95}.  In Puhalskii \cite{Puh91} it was proposed to model large
deviation theory on weak convergence theory and an analogue of
Prohorov's theorem was proved stating  that exponentially tight
sequences of probability measures on metric spaces are Large Deviation (LD) 
relatively  compact, see
also  Puhalskii \cite{Puh93}.
That result opened up a direction of research that
attempts to approach the task of proving an LDP  by using
global relations for the processes involved rather than by 
directly verifying
the  bounds in the definition.
The  idea is to
identify    the deviation function of an exponentially tight sequence
 through equations that arise as limits 
of
 equations satisfied by the stochastic processes in question, e.g., by
 taking a limit in a martingale problem, cf.
Puhalskii \cite{Puh01,Puh16}.

In relation to invariant measures,
it was proposed  in Puhalskii \cite{Puh03}  to  
 take LD limits 
in  the invariance relations.
By definition, an invariant measure can be written as
 a mixture of transient distributions
against the invariant measure itself so that,
  given a stationary stochastic process
$\overline X_n(t)$ with invariant measure $\mu_n$\,, for a bounded continuous
nonnegative function $f$ on the state space,
\begin{equation}
  \label{eq:29}
  \int f(x)^n\mu_n(dx)=\int\int f(x)^nP(\overline X_n(t)\in
  dx|\overline X_n(0)=y)\mu_n(dy)\,.
\end{equation}
Loosely speaking, 
if the transient distribution $P(\overline X_n(t)\in
  dx|\overline X_n(0)=y)$ LD converges to certain deviability $\Pi_{y,t}(x)$ (or
  obeys the LDP with deviation function $I_{y,t}(x)=-\ln\Pi_{y,t}(x)$ for that matter)\footnote{A review of the basics
 of
    Large Deviation (LD) convergence, which is just a different form
    of
the  LDP, and of idempotent processes, can be found in Puhalskii
\cite{Puh23,Puh23a}} and $\mu_n$ LD converges to
  deviability $\Pi$\,, then \eqref{eq:29} implies that
  \begin{equation}
    \label{eq:33'}
    \sup_x f(x)\Pi(x)=\sup_y\sup_x f(x)\Pi_{y,t}(x)\Pi(y)\,.
  \end{equation}
If $\Pi_{y,t}(x)$ converges to  limit $\hat \Pi(x)$\,, as
$t\to\infty$\,, that does not depend  on $y$\,, then letting $t\to\infty$
in \eqref{eq:33'} and noting that $\sup_y\Pi(y)=1$
implies, modulo certain technical conditions,
that $\sup_x f(x)\Pi(x)=\sup_x f(x)\hat\Pi(x)$ so that
$\Pi(x)=\hat\Pi(x)$\,.
Thus, 
\begin{equation}
  \label{eq:20}
  \Pi(x)=\lim_{t\to\infty}\Pi_{y,t}(x)\,. 
\end{equation}

A number of issues arise. Firstly, the exponential tightness of $\mu_n$
needs to be proved. 
The second issue concerns passing to an LD
 limit in the mixture on the righthand side of \eqref{eq:29} in order
 to obtain \eqref{eq:33'}. 
Essentially, continuous LD convergence is needed for the
  transient distributions  (or continuous
LDP for that matter), i.e.,
  that the
distribution
$P(\overline X_n(t)\in
  dx|\overline X_n(0)=y_n)$ LD converges to $\Pi_{y,t}(x)$ whenever $y_n\to
  y$\,, provided $\Pi(y)>0$\,. 
The third issue is  whether
the time limit of $\Pi_{y,t}(x)$ exists and does not depend on $y$\,. It holds provided
 the limit dynamical system for $\overline X_n(t)$ has  a unique 
 equilibrium and the trajectories of the system
are attracted to that equilibrium. In such a case,
\begin{equation}
  \label{eq:35'}
  \Pi(x)=\lim_{t\to\infty}\Pi_{O,t}(x)\,,
\end{equation}
where $O$ represents the  equilibrium.
The latter relation can be written in terms of deviation functions as
\begin{equation}
  \label{eq:36}
  I(x)=\inf_tI_{O,t}(x)\,,
\end{equation}
the righthand side representing the celebrated quasipotential of
Freidlin and Wentzell \cite{wf2},
where $I_{O,t}(x)=-\ln \Pi_{O,t}(x)$\,.
 If the uniqueness of the equilibrium
 does not hold, then 
  more sophisticated balance equations may be needed, see
 Puhalskii \cite{Puh23a}.

In this paper, in the same vein as developments in 
 Kaspi and Ramanan \cite{KasRam13} and
Aghajani and Ramanan \cite{AghRam20}, 
the state variable $Q_n(t)$ 
is appended with residual service times of customers in
service.
 A key distinction is that the
residual service times are encapsulated in an empirical distribution
function,
which device enables one to keep low the dimensions of the processes
involved. As the distribution of the process $X_n$ is of interest, 
the continuous LD convergence is needed for the distribution 
$P(X_n(t)\in dx|(X_n(0),\Phi_n(0)))$ where $\Phi_n(0)$ is a
properly centred and normalised version of the empirical cdf
 of  residual service times.
 The role of the initial condition  $y_n$ is played by
  a pair $(x_n,\phi_n)$\,. The variable $x_n$ accounts for the
initial number of customers and $\phi_n=(\phi_n(t)\,,t\in\R_+)$
is a
centred and normalised empirical cdf of the residual service times of 
 customers in service initially.
The initial condition     is therefore  functional as
$\phi_n$ is an element of the Skorohod space $\D(\R_+,\R)$\,.
 It is essentially
  the  process of the number of remaining customers out of those
  in service initially. 
The convergence of the  initial
 conditions is understood as convergence of 
$(x_n,\phi_n)$ to $(x,\phi)$ in $\R\times \D(\R_+,\R)$\,,
 as $n\to\infty$\,. The 
 condition that $\Pi(y)>0$ translates into the condition that 
$\Pi(x,\phi)>0$ for any LD limit point $\Pi$ of the stationary distribution of
 $(X_n(0),\Phi_n(0))$\,.
It is proved in the paper that
under the invariant measure the
number of customers present and 
the empirical cdf of residual service
times, when properly centred and normalised,
 are  exponentially tight as  random elements of $\R$ and 
Skorohod space $\D(\R_+,\R)$\,, respectively. (As a byproduct, the cdf
of the stationary residual service time converges to $F^{(0)}$
 that is the distribution of a delay in an
equilibrium renewal process with the service time distribution as the
generic interarrival time distribution, see Remark \ref{re:eq_lim}.)
 It may seem  remarkable that the
stationary residual service time empirical cdf (when appropriately
centred and normalised) proves to be
$\bbC_0$--exponentially tight in a fairly general situation
\footnote{ $\bbC_0$ represents the subset of
$\mathbb C$ of functions that vanish at infinity}, i.e., the
idempotent processes that may arise as its
LD limits 
 are continuous and vanish over time. 
Thus, in order to obtain a version of \eqref{eq:33'}, one needs trajectorial
LD convergence when the initial condition
 $(X_n(0),\Phi_n(0))$  LD converges to
some $(x,\phi)\in\R\times\bbC_0$\,. 
Available results on trajectorial moderate deviations
are limited to quite special
initial conditions. In  Puhalskii \cite{Puh23}, it was assumed that the
residual service times of customers initially in service are independent and
have cdf $F^{(0)}$\,. An extension
to the more general initial conditions is the first order of business.
Fortunately, virtually the same proof as in Puhalskii \cite{Puh23}
applies.

After that is taken care of, 
 a version of \eqref{eq:20}  is established,  with $y=(x,\phi)$\,.
It is proved first by using monotonicity considerations
 that \eqref{eq:35'} holds with  $O=(-\beta,\mathbf0)$\,,
where $\mathbf 0 $ represents the function from $\D(\R_+,\R)$ that is
identically equal to $0$\,.
Then,
 the limit idempotent process $ X$ that starts at $(x,\phi)\in\R\times
 \bbC_0$  
is coupled
 with the limit idempotent process that starts at $(-\beta,\mathbf
0)$\,. 
Since $X$ is an idempotent process with  time delay,
 constructing the coupling
takes keeping $X$  in a neighbourhood
of
$-\beta$ for an extended period of time  in order for the past
not to
matter too much. Yet, not always is exact coupling 
 achieved. To have
 the idempotent processes in question
get close to each other proves to be good enough.
In the end,
 the limits of $\Pi_{(x,\phi),t}$ and of  $\Pi_{(-\beta,\mathbf0),t}$
are shown to agree, as $t\to\infty$\,.
The property that $\phi$ vanishes
 over time is essential.
This completes the check that the time limit of $\Pi_{(x,\phi),t}(y)$ exists and does not depend on $(x,\phi)$\,.

Checking  exponential tightness  constitutes the  most 
 challenging part of the proof and is responsible for most of the
 moment conditions in the statement of the main result. 
As mentioned, the exponential tightness of $X_n$ as a random element of a
Skorohod space was addressed in
Puhalskii \cite{Puh23} for the
special initial condition. This paper extends the 
property 
to more general initial conditions, to 
residual service times and to  stationary distributions.
Similarly to  Puhalskii \cite{Puh23},
 proofs use
 the exponential analogue of Lenglart's domination property
in Puhalskii \cite{Puh01}.
The proof of the exponential tightness of the stationary distribution of
 $X_n$ employs a general upper
 bound  in Gamarnik and Goldberg \cite{GamGol13} and
 a lower  bound  that is furnished
 by an infinite
server queue.
 Markov's  inequality is used extensively,
so,  higher--order moments of counting renewal processes
 need to be controlled. 
The available body of literature 
focuses on moments of fixed order and 
emphasises 
dependence of moment  bounds on time, whereas dependence  on the
moment's order is overlooked.
Proving the exponential tightness,
however,
  requires dealing with moments of orders
going to infinity, so, dependence  on the moment's order 
becomes significant. This paper fills the void  by
providing moment bounds with 
explicit dependence   on both   order and time.

The paper's organisation 
is as follows.
In Section
\ref{sec:fluid-stoch-appr}, a precise specification of the model 
is provided and
 the main result on
the LDP for the stationary distribution of $X_n$ is stated.
Subsequent sections use the terminology of LD convergence and
idempotent processes. In Section \ref{sec:ld-conv-skor}, the extension of 
the trajectorial LD convergence of $X_n$
is addressed.
Sections \ref{sec:expon-tightn-stat} and \ref{sec:residual} deal
with the  exponential
tightness of the stationary distributions of the number of
customers and of the residual service times, respectively. 
In Section  \ref{sec:moder-devi-stat}, the idempotent process coupling is
constructed and  the main result is proved. 
The bounds on  the higher order moments of counting renewal processes
are developed in Appendix \ref{sec:higher}.
Appendices  \ref{renewal}, \ref{c_0}
 and \ref{cont_conv} are concerned with 
a nonlinear renewal equation satisfied by $X$\,, 
with   $\bbC_0$--exponential tightness  and with
large deviations of mixtures, respectively, which are needed
 for the proofs.

\section{The LDP for the stationary number-in-the-system}
\label{sec:fluid-stoch-appr}
Assume as  given a sequence of many--server queues indexed by $n$, where $n$
represents the number of  servers, which are assumed homogeneous.
 Arriving customers are served on a first--come--first--served basis.  If upon a customer arrival, there are available
servers, the customer starts being served by one of the available
servers, chosen at random. Otherwise, the customer joins the queue
and awaits their turn to be served. On completing service, the customer
leaves, thus relinquishing the server.
Let $Q_n(t)$ denote the number of customers present at time
$t$\,. Out of those, $Q_n(t)\wedge n $
customers are in service and $(Q_n(t)-n)^+$ customers are
in the queue.  
Service times of  customers initially in the queue   and  service times
of exogenously arriving  customers   are denoted
 by
 $\eta_1,\eta_2,\ldots$ in  the
 order of  entering service and  come from an
iid sequence of positive unbounded random variables with
continuous cdf
$F$\,.  In terms of $F$\,, it is thus assumed that
\begin{equation}
  \label{eq:45}
F(0)=0\,, F(x)<1\,, \text{ for all }x\,.
\end{equation}
The mean service time $\mu^{-1}=\int_0^\infty x\,dF(x)$ is assumed to
be finite.
 Residual service times of customers in service initially
 come from a sequence of  positive random variables
$\eta^{(0)}_1,\eta^{(0)}_2,\ldots$\,.
The number of exogenous arrivals
by $t$ is denoted by  $A_n(t)$, with $A_n(0)=0$\,. It is assumed that $A_n(t)$ has unit jumps.
All stochastic processes are assumed to have
 rightcontinuous paths with lefthand limits.
The random entities $Q_n(0)$\,, $\{\eta_1,\eta_2,\ldots\}$\,,
$\{\eta^{(0)}_1,\eta^{(0)}_2,\ldots\}$ and $A_n=(A_n(t)\,,t\in\R_+)$ are
assumed  independent. 

Let 
${{\hat A}_n}(t)$ denote the number of customers that enter service after time
 $0$ and by time 
$t$\,, with 
${{\hat A}_n}(0)=0$\,.  Since the distribution function
 of  $\eta_i$ is continuous, $\hat A_n$ has
unit jumps a.s.
Balancing  arrivals and departures yields the equation
\begin{equation}
             \label{eq:107}
    Q_n(t)=Q_n^{(0)}(t)+(Q_n(0)-n)^+
+A_n(t)-\int_0^t\int_0^t \mathbf{1}_{\{x+s\le t\}}
\,d\,\sum_{i=1}^{{{\hat A}_n}(s)}\mathbf{1}_{\{\eta_i\le x\}}\,,
\end{equation}
where
\begin{equation}
  \label{eq:1a}
Q_n^{(0)}(t)=  \sum_{i=1}^{Q_n(0)\wedge n}\,\ind{\eta^{(0)}_i> t}\,,
\end{equation}
which quantity represents the number of customers present at time $t$ out
of those in service initially and
\begin{equation}
  \label{eq:79}
  \int_0^t\int_0^t \mathbf{1}_{\{x+s\le t\}}
\,d\,\sum_{i=1}^{{{\hat A}_n}(s)}\mathbf{1}_{\{\eta_i\le x\}}=
\sum_{i=1}^{\hat A_n(t)}\ind{\eta_i+\hat\tau_{n,i}\le t},
\end{equation}
which quantity represents the number of customers that  enter service after  time $0$
and  depart by time $t$\,, $\hat\tau_{n,i}$ denoting the $i$th jump
time of $\hat A_n$\,, i.e.,
\begin{equation}
  \label{eq:24}
  \hat\tau_{n,i}=\inf\{t:\,\hat A_n(t)\ge i\}\,.
\end{equation}
 In addition, since all customers
that are either in the queue initially  or have arrived exogenously by
time $t$  must  either be in the queue at time $t$
 or have entered service by time $t$\,,
\begin{equation}
    \label{eq:40}
 (Q_n(0)-n)^++A_n(t)=(Q_n(t)-n)^++{\hat A}_n(t)\,.
\end{equation}
 For the existence and uniqueness of solutions to
\eqref{eq:107} and \eqref{eq:40}, the reader is referred to
 Puhalskii and Reed
\cite{PuhRee10}.

If arrivals at the $n$th  queue occur according to a renewal process,
if the system is underloaded\,, i.e., $\rho_n<1$\,, and 
if the interarrival 
distribution is spread out, then there exists a
unique stationary distribution of $Q_n(t)$\,,
which is also the longterm distribution for every initial state, see, e.g.,
 Corollary 2.8 on p.347 in Asmussen \cite{Asm03}.
Let random variable $ Q^{(s)}_{n}$ be distributed according to 
  the stationary distribution of $Q_n(t)$ and let 
  \[
     X^{(s)}_{n}=\frac{\sqrt{n}}{b_n}\,( \frac{Q^{(s)}_{n}}{n}-1)\,.
  \]
Given $q=(q(t)\,,t\in\R_+)\in\D(\R_+,\R)$\,, let
\begin{equation}
   \label{eq:2}
         I^Q(q)=\frac{1}{2}\inf\{
\int_0^\infty \dot w(t)^2\,dt
+\int_0^\infty\int_0^1 \dot k(x,t)^2\,dx\,dt\}\,,
 \end{equation}
the $\inf$ being taken over $
w=(w(t)\,,t\in\R_+)\in\D(\R_+,\R)$ and
$k=((k(x,t)\,,x\in[0,1])\,,t\in\R_+)\in\D(\R_+,\D([0,1],\R_+))$ such that
$w(0)=0$\,, $k(x,0)=k(0,t)=k(1,t)=0$\,,
 $w$ and $k$ are absolutely continuous with respect to Lebesgue measures
on  $\R_+$ and $[0,1]\times \R_+ $\,, respectively, and, for
all $t$\,, 
\begin{multline}\label{eq:1}
q(t)=
-\beta
 +\int_0^tq(t-s)^+\,dF(s)
+\int_0^t\bl(1-F(t-s)\br)\sigma\,\dot w(s)\,ds\\
+\int_{\R_+^2} \ind{x+s\le t}\,\dot k(F(x),\mu s)\,dF(x)\,\mu \,ds
\,,
\end{multline}
 where 
$\dot w(t)$ and $\dot k(x,t)$
 represent the respective Radon--Nikodym derivatives.
If  
$w$ and $k$ as indicated do not exist, then $I^Q(q)=\infty$\,.
 By Lemma B.1 in
Puhalskii and Reed \cite{PuhRee10}, the equation 
\eqref{eq:1} has a unique solution for $q(t)$
as an element of $\mathbb L_{\infty,\text{loc}}(\R_+,\R)$\,.
Furthermore, the function $q$ is seen to be continuous, it is
absolutely continuous with respect to Lebesgue measure 
provided so is the function $F$\,.

Say  that, given real--valued sequence 
$r_n\to\infty$\,, as $n\to\infty$\,,   sequence $P_n$
of probability laws on the Borel $\sigma$--algebra of
  metric space $M$ and   $[0,\infty]$--valued function $I$
 on $M$ such that the sets $\{y\in M:\,I(y)\le \gamma\}$ are compact
 for all $\gamma\ge0$\,, the sequence $P_n$  obeys
the LDP for rate $r_n$ with  deviation function $I$ provided
$  \lim_{n\to\infty}1/r_n\,\ln P_n(W)=-\inf_{y\in W}I(y)\,,
$ for every Borel set $W$ such that the infima of $I$ over the
interior  and  closure of $W$ agree.

The main result is stated next. Let $\xi_n$ represent a generic
interarrival time for the $n$th queue and let $\eta$ represent a generic
service time.
\begin{theorem}
  \label{the:inv}
Suppose that
 \begin{enumerate}
\item 
each $A_n$ is a  renewal process of rate $\lambda_n$ whose
interarrival time distribution is spread out,
\item
 the sequence of processes
$\bl((A_n(t)-\lambda_nt)/(b_n\sqrt{n}),\,t\in\R_+\br)
$
obeys the LDP in $\D(\R_+,\R)$  for rate $b_n^2$\,, as $n\to\infty$\,,
with deviation function 
$I^A(a)$ such that $I^A(a)=1/(2\sigma^2)\int_0^\infty\dot a(t)^2\,dt$\,, 
provided $a=(a(t)\,,t\in\R_+)$ is an absolutely continuous nondecreasing function with
$a(0)=0$\,, and
$I^A(a)=\infty$\,, otherwise, where $\sigma>0$\,,
\item the heavy traffic condition \eqref{eq:105} holds with $\beta>0$\,,
\item $\sup_nE(n\xi_n)^{2}<\infty$\,,
$\sup_n n^{1/2}/b_n
(E\xi_n^{b_n^2+1})^{1/b_n^2}<\infty$\,,
$\sup_nb_nn^{1/2}(E\xi_n^{b_n^2})^{1/b_n^2}<\infty$\,,
 \item $  \sup_nb_nn^{-1/2}
(E\eta^{b_n^2+1})^{1/b_n^2}<\infty\,,$
\item 
$\sup_nn^{1/b_n^2}<\infty$\,, $ \lim_{n\to\infty}  b_n^6n^{-1}=0$\,.
\end{enumerate}
Then
 the sequence $\{X^{(s)}_{n}\,,n\in\N\}$
  obeys the LDP in $\R$ for rate $b_n^2$ with
  deviation function
  \[
     I^{(s)}(x)=\inf I^Q(q)\,,
  \]
 the $\inf$ being taken over trajectories $q=(q(t)\,,t\in\R_+)$ 
such that
$q(0)=-\beta$ and $q(t)\to x$\,, as  $t\to\infty$\,. 
\end{theorem}
\begin{remark}
According to the heavy traffic condition \eqref{eq:105}, $\lambda_n$
is of order $n$ so that the expectation of
$\hat \xi_n=n\xi_n$ is of order one.
The  higher--order
 moment conditions on $\xi_n$ in part 4 take the form
$\sup_n b_n^{-1}n^{-1/2}
(E\hat\xi_n^{b_n^2+1})^{1/b_n^2}<\infty$
and $\sup_nb_nn^{-1/2-1/b_n^2}(E\hat\xi_n^{b_n^2})^{1/b_n^2}<\infty$\,,
  which do not seem  to be
 too restrictive.
 For instance, if
$\hat\xi_n$ is exponential of rate 1
(hence, $\xi_n$ is exponential of rate  $n$)\,, then
$E\hat\xi_n^{p}= \Gamma(p+1)$
so that both conditions hold provided
$\sup_nb_n^6 n^{-1}<\infty$\,.
If $\eta$ is exponential with $\mu$\,, an analogous condition
 works for part 5. The first condition in part 6 stands out in that it
 requires  $b_n$ not to grow too slowly as compared 
with $n$ whereas the rest of
 the conditions put constraints from above on the growth rate of
 $b_n$\,. Still, the conditions seem mutually compatible,
 e.g., in the exponential case one can take
 $b_n= n^\epsilon$\,, where $0<\epsilon< 1/6$\,.
It is also noteworthy that the second condition in part 6 implies
that $b_n/\sqrt{n}\to0$\,.
\end{remark}
\section{Trajectorial LD convergence}
\label{sec:ld-conv-skor}
This Section is concerned with the extension of Theorem 2.2 in Puhalskii
\cite{Puh23} to more general initial conditions. 
 Let
 $X_n=(X_n(t),t\in\R_+)$, where, as defined earlier,
\begin{equation}
    \label{eq:30}
X_n(t)=\frac{\sqrt{n}}{b_n}\bl(\frac{Q_n(t)}{n}-1\br)\,.
\end{equation}
The stochastic process $X_n$ is considered as a random element of
$\D(\R_+,\R)$\,,
 which 
 is endowed with a
 metric rendering it a  complete separable metric space, see Ethier and
 Kurtz \cite{EthKur86}, Jacod and Shiryaev \cite{jacshir}, for
 more detail.

Let \begin{equation}
  \label{eq:73}
  F^{(0)}(x)=\mu\int_0^x(1-F(y))\,dy\,.
\end{equation}
It is  the cdf  of the delay in an equilibrium renewal
 process with generic inter--arrival cdf $F$\,,
 see, 
Section 5.3 in Asmussen \cite{Asm03},
Section 10.3 in Grimmett and Stirzaker \cite{GriSti01}, or
Section 3.9 in Resnick \cite{Res02}
 for more detail.
  Introduce
\begin{equation}
  \label{eq:46}
  Y_n(t)=\frac{\sqrt{n}}{b_n}\,\bl(\frac{A_n(t)}{n}-\mu t\br)\,,
\end{equation}
\begin{align*}
    U_n(x,t)&=\frac{1}{b_n\sqrt{n}}\,\sum_{i=1}^{{{\hat A}_n}(t)}
\bl(\ind{\eta_i\le x}-F(x)\br)
\intertext{ and }
  \Theta_n(t)&=-
\int_{\R_+^2} \mathbf{1}_{\{x+s\le t\}}\,dU_n(x,s)\,.\notag
\end{align*}
As, owing to \eqref{eq:73}, 
$  1-F^{(0)}(t)+\mu t-\mu\int_0^t(t-s)dF(s)=1\,,
$
 \eqref{eq:107}, \eqref{eq:1a}, \eqref{eq:30}, and
\eqref{eq:46}  imply, after some algebra, that
  \begin{multline}
      \label{eq:67}
  X_n(t)=(1-F(t))X_n(0)^+-(1-F^{(0)}(t))X_n(0)^-+\Phi_n(t)\\+
\int_0^tX_n(t-s)^+\,dF(s)
+Y_n(t)-\int_0^t
Y_n(t-s)\,dF(s)+\Theta_n(t)
\,,
\end{multline}
where  $x^-=\max(-x,0)$ and 
\begin{equation}
  \label{eq:14}
\Phi_n(t)=  \frac{1}{b_n\sqrt n}\,\sum_{i=1}^{Q_n(0)\wedge
    n}(\ind{\eta_i^{(0)}> t}-(1-F^{(0)}(t)))\,.
\end{equation}
(Note that $\Phi_n(0)=0$ by the $\eta_i^{(0)}$ being positive.)
For the purpose of proving a steady-state LDP, the process 
$ \Phi_n=( \Phi_n(t),\,t\in\R_+)$ needs to be more general than in
\eqref{eq:14}, so, in the next theorem it is assumed that 
$\Phi_n$  is a stochastic process which has trajectories
from $\D(\R_+,\R)$ with
$\Phi_n(0)=0$\,.
The existence and uniqueness for $X_n(t)$ solving \eqref{eq:67}\,, 
given $X_n(0)$\,,
$\Phi_n(t)$\,, $Y_n(t)$ and $\Theta_n(t)$\,, follows from Puhalskii and
Reed \cite{PuhRee10}.
Let 
$Y_n=(Y_n(t),\,t\in\R_+)$\,. 
\begin{theorem}
  \label{the:heavy_traffic}
  \begin{enumerate}
  \item 
Suppose that $X_n$ solves \eqref{eq:67}. Suppose that
$X_n(0)$\,,    $\Phi_n$ and $Y_n$  jointly LD  converge in
 distribution   in $\R\times\D(\R_+,\R)^2$ at rate $b_n^2$\,,
  as $n\to\infty$, to 
$X(0)$\,, $\Phi$ and $Y$\,, respectively, where $X(0)\in\R$\,,
$\Phi=(\Phi(t)\,,
 t\in\R_+)$
and $Y=(Y(t)\,,t\in\R_+)$ 
are idempotent processes with  continuous trajectories.
Let the heavy traffic condition \eqref{eq:105}  hold. Let
\begin{equation}
  \label{eq:4}
   b_n^6n^{1/b_n^2-1}\to0\,, \text{ as } n\to\infty\,.
\end{equation}

Then   the sequence $X_n$
LD converges in distribution in $\D(\R_+,\R)$ at rate $b_n^2$ to 
    the idempotent process
$X=(X(t),t\in\R_+)$ that is the unique  solution to  the equation
\begin{multline}
  \label{eq:70}
         X(t)=(1-F(t))X(0)^+- (1-F^{(0)}(t))X(0)^{-}+\Phi(t)
+\int_0^tX(t-s)^+\,dF(s)
\\
+Y(t)-\int_0^t Y(t-s)\,dF(s)
+\int_0^t\int_{0}^t \ind{x+s\le t}\,\dot K(  F(x),\mu s)\,dF(x)\,\mu\,ds
\,,
 \end{multline}
where
 $K=(K(x,t),
(x,t)\in   [0,1]\times \R_+)$ is a  Kiefer idempotent process,  
$(X(0),\Phi)$,
$Y$\,,  and
$K$ being
independent.
Furthermore, the sequence $(X_n(0),\Phi_n,Y_n,X_n)$ LD converges in
$\R\times \D(\R_+,\R^3)$ to $(X(0),\Phi,Y,X)$\,.
\item Under the hypotheses  of Theorem
 \ref{the:inv},
 $   Y(t)=\sigma W(t)-\beta\mu t$\,, with $W=(W(t)\,,t\in\R_+)$ being a standard
 Wiener idempotent process that is independent of 
$X(0)$\,, $\Phi$ and $K$\,.  The equation \eqref{eq:70} takes the
form
\begin{multline*}
            X(t)=(1-F(t))X(0)^+ -(1-F^{(0)}(t))X(0)^-+\Phi(t)
-\beta F^{(0)}(t)\\
+\int_0^tX(t-s)^+\,dF(s)
+\sigma\int_0^t\bl(1-F(t-s)\br)\,\dot W(s)\,ds
\\+\int_0^t\int_{0}^t \ind{x+s\le t}\,\dot K(F(x),\mu s)\,dF(x)\,\mu \,ds
\,.
\end{multline*}
 \end{enumerate}
\end{theorem}
 \begin{remark}
    In the special case where \eqref{eq:14} holds 
and     the random variables $\eta^{(0)}_i$ are
iid with cdf $F^{(0)}$\,,
$  \Phi(t)=W^0(F^{(0)}(t))$\,,
where 
$(W^0(x)\,,x\in[0,1])$ is a Brownian bridge idempotent process
and
  $X(0)$,
$W$, $W^0$ and
$K$ are
independent. 
 This is part 2 of Theorem 2.2 in Puhalskii \cite{Puh23}.
\end{remark}
The proof of Theorem \ref{the:heavy_traffic} procedes along similar lines as  the proof of Theorem 2.2 in Puhalskii
\cite{Puh23}, the difference being
the contents of Lemma 3.3 in Puhalskii \cite{Puh23}. 
The analogue of that lemma with $\Phi_n$ as $X_n^{(0)}$ and with $\Phi$ as
$X^{(0)}$ needs to be used. The proof of the new version is analogous
to and is
actually simpler than the proof in  Puhalskii \cite{Puh23} because the
LD convergence of $\Phi_n$ to $\Phi$ is built into  the hypotheses unlike
the LD convergence of $X_n^{(0)}$ to $X^{(0)}$ in  Puhalskii \cite{Puh23}
which had to be argued.
That the condition \eqref{eq:4} suffices follows from a more careful
analysis at the final stages of the proof of Lemma 3.1 in
Puhalskii \cite{Puh23}.
\section{Exponential tightness of the stationary number of customers}
\label{sec:expon-tightn-stat}

This section  is concerned with    exponential tightness
 of the distribution of $X^{(s)}_n$
under the hypotheses of Theorem \ref{the:inv}.
The random variable $X^{(s)}_n$ being the limit in distribution
of $X_n(t)$ for any initial condition, as $t\to\infty$\,,
the initial condition  in Theorem 3  
in Gamarnik and Goldberg \cite{GamGol13} is particularly useful.
It is recapitulated next.

Let $A'_n$ represent an equilibrium version of $A_n$\,.
It is a delayed renewal process 
with the same generic inter--arrival distribution as $A_n$\,,
whose epoch of the first renewal has the cdf
$(\lambda_n\int_0^xP(\xi_{n}>y)\,dy\,,x\in\R_+) $\,. 
Similarly, let
  $N'_1,\ldots,N'_n$ represent iid equilibrium
 renewal processes with inter--arrival 
cdf  $F$\,. Suppose that the processes $A'_n$ and $N'_1\,,\ldots,N'_n$
are independent. 
Suppose that arrivals occur according to $A'_n$\,,
 that initially there are $n$ customers in service, whose
  residual service times are mutually independent as well as independent of
   arrivals and are
distributed according to $F^{(0)}$\,, 
and that  there is no initial queue,  i.e., $Q_n(0)=n$\,.
Let $Q'_n(t)$ denote the  number of customers present at $t$ under
these hypotheses and let 
\begin{equation}
  \label{eq:131}
  X'_n(t)=\frac{\sqrt n}{b_n}\bl(\frac{Q'_n(t)}{n}-1\br)\,.
\end{equation}
Similarly, the use of primes for related
 random variables will signify that  this special setup is assumed.

Denote, for  random variable $\zeta$\,,
$\norm{\zeta}_{n}=(E\abs{\zeta}^{b_n^2})^{1/b_n^2}$
and $\norm{\zeta}_{n,\theta}=(E\abs{\zeta}^{\theta
  b_n^2})^{1/(\theta b_n^2)}$\,, where
$\theta>0$\,. 
The next result is the centrepiece of this section.
 \begin{theorem}\label{the:stat_exp_tight}
Under the hypotheses of Theorem \ref{the:inv}, for some $\theta>0$\,,
\begin{equation}
  \label{eq:134}
  \limsup_{n\to\infty}\sup_{t\in\R_+}\norm{X'_n(t)}_{n,\theta}<\infty\,.
\end{equation}
 \end{theorem}
 \begin{remark}
   \label{re:stat_exp}
Clearly, under \eqref{eq:134}, 
\begin{equation}
  \label{eq:23}
  \limsup_{n\to\infty}\norm{X_n^{(s)} }_{n,\theta}<\infty\,,
\end{equation}
implying that
  the sequence of the distributions of $ X^{(s)}_{n}$ is exponentially
   tight of order $b_n^2$ in $\R$\,.
 \end{remark}
The proof of Theorem \ref{the:stat_exp_tight} examines separately 
$\norm{X'_n(t)^- }_{n}$
 and 
$\norm{X'_n(t)^+}_{n,\theta}$\,.
The next lemma addresses the easy  part.
\begin{lemma}
  \label{le:negative part}
Under the hypotheses of Theorem \ref{the:inv}, 
\[
    \limsup_{n\to\infty}\sup_{t\in\R_+}\norm{X_n'(t)^- }_{n}<\infty\,.
\]
\end{lemma}
\begin{proof}
  
  The proof  uses a lower bound on $Q_n'(t)$ that
is the number--in--the--system in the associated infinite server queue.
Let
 $\tau'_{n,i}$ represent the $i$th arrival epoch for
$A'_n$\,, i.e., in analogy with \eqref{eq:24},
\[
  \tau'_{n,i}=\inf\{t:\, A'_n(t)\ge i\}\,.
\]
By \eqref{eq:40} and the fact that $Q_n'(0)=n$\,,
$A'_n(t)\ge \hat A'_n(t)$ and   $\hat\tau'_{n,i}\ge
\tau'_{n,i}$\,, for $i\ge 1$\,, so, by \eqref{eq:107} and  \eqref{eq:79},
\[
      Q'_n(t)
\ge  Q_n'^{(0)}(t)+
A'_n(t)-\sum_{i=1}^{
A'_n(t)}\ind{\tau'_{n,i}+\eta_i\le t}\,.\]
Let
\begin{equation}
  \label{eq:89}
  \breve Q_n(t)=Q_n'^{(0)}(t)+A'_n(t)
-\int_0^t\int_0^t \mathbf{1}_{\{x+s\le t\}}
\,d\,\sum_{i=1}^{ A'_n(s)}\mathbf{1}_{\{\eta_i\le x\}}
\end{equation}
and \[   \breve X_n(t)=\frac{\sqrt{n}}{b_n}\bl(\frac{\breve Q_n(t)}{n}-1\br)\,.
\]
It suffices to prove that  \begin{equation}
  \label{eq:113}
  \limsup_{n\to\infty}\sup_{t\in\R_+}
\,\norm{\breve X_n(t)}_n <\infty\,.
\end{equation}
An algebraic manipulation of \eqref{eq:89} that is analogous to the
derivation of \eqref{eq:67}
yields
\[  \breve X_n(t)   =
\Phi'_n(t)+
H'_n(t)-\Theta'_n(t)\,,
\]
where
\begin{equation}
  \label{eq:94}
  \Phi'_n(t)=\frac{1}{b_n\sqrt n}\,\sum_{i=1}^{
    n}(\ind{\eta_i'^{(0)}> t}-(1-F^{(0)}(t)))\,,
\end{equation}
\begin{equation}
  \label{eq:90}
    H'_n(t)=Y_n'(t)-\int_0^t
Y'_n(t-s)\,dF(s)\,,
\end{equation}
\begin{equation}
  \label{eq:90a} Y_n'(t)
=\frac{\sqrt{n}}{b_n}\,\bl(\frac{A'_n(t)}{n}-\mu t\br)\,,
\end{equation}
\begin{equation}
    \label{eq:144}
     \Theta'_n(t)=
\int_{\R_+^2} \mathbf{1}_{\{x+s\le t\}}\,d U'_n(x,s)
\end{equation}
and
\begin{equation}
  \label{eq:143}
      U'_n(x,t)=\frac{1}{b_n\sqrt{n}}\,\sum_{i=1}^{{{ A}'_n}(t)}
\bl(\ind{\eta_i\le x}-F(x)\br)\,.
\end{equation}
The inequality  \eqref{eq:113} is proved by proving that

\begin{equation}
  \label{eq:98}
\limsup_{n\to\infty}\sup_{t\in\R_+}
\norm{\Phi'_n(t)}_n<\infty\,,
   \end{equation}
\begin{equation}
  \label{eq:153}
      \limsup_{n\to\infty}\sup_{t\in\R_+}
\norm{H'_n(t)}_n <\infty
\end{equation}
and that
\begin{equation}
  \label{eq:145}
    \limsup_{n\to\infty}\sup_{t\in\R_+}
\norm{\Theta'_n(t)}_n<\infty\,.
\end{equation}
\begin{proof}[Proof of \eqref{eq:98}]
By \eqref{eq:94}, applying
 the bound  (5.6) in the proof of Theorem 19, Chapter
III, \S5 in Petrov \cite{Pet87} \footnote{More specifically, the following
  bound is used. Suppose $X_1,\ldots,X_n$ are independent with $EX_k=0$. 
Then, provided $p\ge2$\,,
\[
  E\abs{\sum_{k=1}^nX_k}^p\le (\frac{p}{2}+1)^p\sum_{k=1}^nE\abs{X_k}^p
+p(\frac{p}{2}+1)^{p/2}e^{p/2+1}(\sum_{k=1}^nEX_k^2)^{p/2}.
\]
},  see also Whittle \cite{Whi60} for similar results,
\begin{multline*}
  E\abs{\Phi'_n(t)}^{b_n^2}= \frac{1}{(b_n\sqrt{n})^{b_n^2}}\,
E\abs{\sum_{i=1}^n(
\ind{\eta'^{(0)}_i\le t}-F^{(0)}(t))}^{b_n^2}\\
\le\frac{1}{(b_n\sqrt{n})^{b_n^2}}\,
\bl((\frac{b_n^2}{2}+1)^{b_n^2}n(F^{(0)}(t)(1-F^{(0)}(t))^{b_n^2}
+(1-F^{(0)}(t))F^{(0)}(t)^{b_n^2})\\+
b_n^2(\frac{b_n^2}{2}+1)^{b_n^2/2}e^{b_n^2/2+1}n^{b_n^2/2}
\bl(F^{(0)}(t)(1-F^{(0)}(t))\br)^{b_n^2/2}\,.
\end{multline*}
Hence,
\[
  \norm{\Phi'_n(t)}_n\le
\frac{1}{b_n\sqrt{n}}\,
\bl((b_n^2+1)n^{1/b_n^2}\\+
b_n^{2/b_n^2}(b_n+1)e^{1/2+1/b_n^2}n^{1/2}\br)\le
\hat C (b_nn^{1/b_n^2-1/2}+1)\,,
\]
for some $\hat C>0$\,, which implies \eqref{eq:98}
due to part 6 of Theorem \ref{the:inv}.
\end{proof}
\begin{proof}[Proof of \eqref{eq:153}]
Introduce
\begin{equation}
  \label{eq:16}
        \overline A'_n(s)=A_n'(s)-\lambda_ns\,.
\end{equation}
    By \eqref{eq:90} and \eqref{eq:90a}, after some algebra,
\begin{multline}
  \label{eq:158}
H'_n(t)
=\frac{\sqrt{n}}{b_n}\,\,
\int_0^t
\frac{\overline A_n'(t)-\overline A_n'(t-s)}{n}
\,dF(s)
+
\frac{\sqrt{n}}{b_n}\,\frac{\overline A_n'(t)}{n}(1-F(t))\\
+\frac{\sqrt{n}}{b_n}\,(\frac{\lambda_n}{n}-\mu)\bl(\int_0^ts\,dF(s)
+t(1-F(t))\br)\,.
\end{multline}
The terms on the latter righthand side are dealt with in order.
By $A_n'$ having  stationary increments, 
$  \int_0^t
(\overline A_n'(t)-\overline A_n'(t-s))\,dF(s)
$ is distributed as
$  \int_0^t
\overline A_n'(s)\,dF(s)\,.$
By Jensen's inequality,
\[
\norm{\int_0^t
\abs{\overline A_n'(s)}\,dF(s)}_n \le
\int_0^t\norm{
\overline A_n'(s)}_n 
\,dF(s)
\,.
\]
By Lemma \ref{le:stat_moments} (with $p=b_n^2$), for some $\tilde C>0$\,,
\begin{multline}
  \label{eq:19}
  \norm{\overline A_n'(s)}_n\le
\tilde C\bl((1+(\lambda_ns)^{1/b_n^2})
(E(\lambda_n\xi_n)^{b_n^2+1})^{1/b_n^2}\\+
(1+(\lambda_ns)^{1/2+1/b_n^2})b_n^{1+2/b_n^2}
(E(\lambda_n\xi_n)^2)^{3/2}\br)\,.
\end{multline}
Hence,
\begin{multline}
  \label{eq:13}
    \int_0^t
\frac{\norm{\overline A_n'(s)}_n}{n}\,dF( s) \le
\frac{\tilde C}{n}\bl(
(E(\lambda_n\xi_n)^{b_n^2+1})^{1/b_n^2}
\int_0^\infty (1+(\lambda_ns)^{1/b_n^2})\,dF(s)\\+
b_n^{1+2/b_n^2}
(E(\lambda_n\xi_n)^2)^{3/2}\int_0^\infty 
(1+(\lambda_ns)^{1/2+1/b_n^2})\,dF(s)\br)\,.
\end{multline}
Since, under the hypotheses, $\lambda_n/n\to\mu$\,,  $b_n^2/n\to0$\,,
$\sup_n(\lambda_n^2E\xi_n^2)<\infty$\,,
 $\sup_nn^{1/b_n^2}<\infty$\,,
$\sup_n n^{1/2}/b_n
(E\xi_n^{b_n^2+1})^{1/b_n^2}<\infty$\,, and
$\int_0^\infty s\,dF(s)<\infty$\,, \eqref{eq:13} implies that
\begin{equation}
  \label{eq:128}
  \limsup_{n\to\infty}\frac{\sqrt{n}}{b_n}\,
\int_0^\infty
\frac{\norm{\overline A_n'(s)}_n}{n}\,dF(s) <\infty\,.
\end{equation}
Also by \eqref{eq:19},   noting
 that $(1-F(t))t\to0$\,, as $t\to\infty$\,,
\[
\limsup_{n\to\infty}
\sup_{t\in\R_+}
 \frac{\norm{\overline A_n'(t)}_n}{n}
\,(1-F(t)) <\infty\,.
\]
Finally, the absolute value of
the last term on the righthand side of \eqref{eq:158} is bounded in
$n$ and $t$ by the heavy traffic condition \eqref{eq:105} and
by $\int_0^\infty
s\,dF(s)$ being finite.
Substituting the bounds in \eqref{eq:158} implies \eqref{eq:153}.

\end{proof}
 \begin{proof}[Proof of \eqref{eq:145}]
    By  \eqref{eq:144} and \eqref{eq:143},
\[
  \Theta'_n(t)=\frac{1}{b_n\sqrt{n}}\,\sum_{i=1}^{A'_n(t)}
(\ind{\eta_i\le t-\tau_{n,i}'}-F(t-\tau'_{n,i}))\,.
\]
Let $\mathcal{F}^{A'_n}$ denote the $\sigma$--algebra generated by
the process $A'_n$\,. By the arrivals and
service being independent,
\[
  E(\abs{\Theta'_n(t)}^{b_n^2}|\mathcal{F}^{A'_n})=
\Bl(\frac{1}{b_n\sqrt{n}}\,\Br)^{b_n^2}E\abs{\sum_{i=1}^N
(\ind{\eta_i\le t- y_i}-F(t-y_i))}^{b_n^2}
\large|_{N=A'_n(t), y_i=\tau_{n,i}'}\,.
\]
Similarly to the proof of \eqref{eq:98}, 
by Petrov's bound  for not necessarily identically distributed
$X_i$\,, 
see  \cite{Pet87},
\begin{multline*}
  E\abs{\sum_{i=1}^N
(\ind{\eta_i\le t- y_i}-F(t-y_i))}^{b_n^2}
\le \bl(\frac{b_n^2}{2}+1\br)^{b_n^2}\sum_{i=1}^N
u_n(t-y_i)
\\+b_n^2\,\bl(\frac{b_n^2}{2}+1\br)^{b_n^2/2}e^{b_n^2/2+1}
\bl(\sum_{i=1}^{N}v(t-y_i)\br)^{b_n^2/2}\,,
\end{multline*}
where \begin{align}
u_n(y)=  E\abs{\ind{\eta\le y}-F(y)}^{b_n^2}
=F(y)(1-F(y))^{b_n^2}+(1-F(y))F(y)^{b_n^2}\notag
\intertext{ and }
  \label{eq:150}
  v(y)=E(\ind{\eta\le y}-F(y))^{2}=F(y)(1-F(y))
\end{align}
so that\begin{multline}
  \label{eq:151}
    E\abs{\Theta'_n(t)}^{b_n^2}\le
\frac{(b^2_n+2)^{b_n^2}}{2^{b_n^2}b_n^{b_n^2}n^{b_n^2/2}}
E\int_0^tu_n(t-s)\,dA'_n(s)\\
+\frac{b_n^2(b_n^2+2)^{b_n^2/2}e^{b_n^2/2+1}}{b_n^{b_n^2}2^{b_n^2/2}}\,E
\bl(\frac{1}{n}\,\int_0^tv( t- s)\,dA'_n(s)\br)^{b_n^2/2}\,.
\end{multline}
Since $A'_n$ is of rate $\lambda_n$ and has stationary increments,
\begin{equation}
  \label{eq:149}
E\int_0^tu_n(t-s)\,dA'_n(s)
=
\int_0^tu_n(t-s)\,dEA'_n(s)
=\int_0^\infty u_n(s)\,\lambda_n\,ds\,.
\end{equation}
The second term on the righthand side of
\eqref{eq:151} is handled next.
By \eqref{eq:150},
\begin{equation}
  \label{eq:165}
   \int_0^tv( t- s)\,dA'_n(s)\le \int_0^t(1-F( t- s))\,dA'_n(s)\,.
\end{equation}
By  $A'_n$ having stationary increments, assuming it has been extended
to the whole real line,
the latter integral is distributed  as
$ \int_0^t(1-F( t- s))\,d_sA'_n(s-t)$\,, where $d_s$ signifies that
variable $s$ is being incremented.
A change of  variables yields
\begin{equation}
  \label{eq:166}
  \int_0^t(1-F( t-
  s))\,d_sA'_n(s-t)
=\int_{0}^t(1-F(s))\,(-dA'_n(-s))\,.
\end{equation}
The process $(-A'_n(-s)\,,s\in\R_+)$ is distributed as the process
$(A'_n(s)\,,s\in\R_+)$ because both are equilibrium renewal
processes with the same generic inter--arrival  distribution.
Hence, the righthand integral in \eqref{eq:166} is distributed as
  $\int_{0}^t(1-F(s))\,dA'_n(s)$\,.
By integration by parts,
\[
  \int_{0}^t(1-F(s))\,dA'_n(s)=(1-F(t))A'_n(t)+
\int_{0}^tA'_n(s)\,dF(s)\,.
\]
Consequently, owing to \eqref{eq:165} and \eqref{eq:166}, by Jensen's
inequality, for $n$ great,
\[
  E
\bl(\frac{1}{n}\,\int_0^tv( t- s)\,dA'_n(s)\br)^{b_n^2}\le
\frac{2^{b_n^2-1}}{n^{b_n^2}}\Bl((1-F(t))^{b_n^2}EA'_n(t)^{b_n^2}+
\,E(\int_{0}^tA'_n(s)\,dF(s))^{b_n^2}\Br)\,.
\]
It follows that
\begin{equation}
  \label{eq:172}
\norm{\frac{1}{n}\,\int_0^tv( t-
s)\,dA'_n(s)}_n\le
\frac{2}{n}\,(1-F(t))\norm{A'_n(t)}_n+
\frac{2}{n}\,\norm{\int_{0}^tA'_n(s)\,dF(s)}_n\,.
\end{equation}
By Jensen's inequality,
\begin{equation}
  \label{eq:164}
\norm{\int_{0}^tA'_n(s)\,dF(s)}_n\le\int_{0}^t\norm{A'_n(s)}_n \,dF(s)\,.
\end{equation}
 Lemma \ref{le:renewal_moments}, as well as
the facts that $\lambda_n/n\to\mu$\,, that $b_n^2/n\to0$ and
that
$\sup_nn^2E\xi_n^2<\infty$\,, yield the existence of $C>0$ such that,
 for  all $t$
and all $n$ great enough,
\begin{equation}
  \label{eq:130}
  \norm{A'_n(t)}_n\le
Cn(1+t)\,,
\end{equation}
which relation implies \eqref{eq:145} when combined with \eqref{eq:151},
 \eqref{eq:149},  \eqref{eq:172}, and \eqref{eq:164},  on recalling
 that $\int_0^\infty s\,dF(s)<\infty$\,.
\end{proof}\end{proof}

Next comes the hard part.
\begin{lemma}
  \label{le:norm_bound}
Under the hypotheses of Theorem \ref{the:inv}, for some $\theta>0$\,,
\begin{equation}
  \label{eq:15}
  \limsup_{n\to\infty}\sup_{t\in\R_+}\norm{X_n'(t)^+ }_{n,\theta}<\infty\,.
\end{equation}
\end{lemma}
\begin{proof}
By  Theorem 3 in Gamarnik and Goldberg \cite{GamGol13}, for $r\in\R_+$\,,
\begin{equation}
  \label{eq:99}
  P((Q'_n(t)-n)^+>r+1)\le 
P\bl(\sup_{s\in[0,t]}(A'_n(s)-\sum_{k=1}^nN'_k(s))>r+1\br)\,.
\end{equation}

At this stage, it is   convenient  to switch back to
an  ordinary renewal process 
as  an arrival
 process. Let $\tilde A_n(t)=A'_n(t+\xi'_{n,1})-1$\,,
 where $\xi'_{n,1}$  represents the time of the first  renewal
 in the
 equilibrium renewal process $A'_n$ and
$t\in\R_+$\,. 
The process 
$\tilde A_n $ is an ordinary renewal process and
$A'_n(t)\le \tilde A_n(t)+1  $\,.
By
\eqref{eq:99} and \eqref{eq:131}, taking into account that
 $\tilde A_n$ is distributed as $A_n$\,,
\begin{equation}
  \label{eq:246}
P(X_n'(t)>r+1)
\le P(\sup_{s\in\R_+}\frac{1}{b_n\sqrt{n}}\,(A_n(s)-
\sum_{k=1}^nN'_k(s))>r)\,.
\end{equation}
The righthand side is being bounded next.
Let $\tau_{n,i}\,, i\in\N\,,$ represent  arrival times for
$A_n$ and let $\tau_{n,0}=0$\,. 
Let
\[
      \overline N'_k(s)=N'_k(s)-\mu s\,,
\]
\[
  \overline\tau_{n,i}=\tau_{n,i}-\frac{i}{\lambda_n}\,,
\]
and 
\[
    \beta_n=\frac{\sqrt{n}}{b_n}\,(1-\rho_n)\,.
\]

 Then, drawing on Gamarnik and Goldberg \cite{GamGol13}
on the next line,
\begin{multline}
  \label{eq:31}
  \sup_{s\in\R_+}(A_n(s)-\sum_{k=1}^nN'_k(s))=
\sup_{i\in\N}(A_n(\tau_{n,i})-\sum_{k=1}^nN'_k(\tau_{n,i}))
=\sup_{i\in\N}(i-\sum_{k=1}^nN'_k(\tau_{n,i}))
\\
=\sup_{i\in\N}\bl(i(1-\frac{1}{\rho_n})-
\sum_{k=1}^n\overline N'_k(\tau_{n,i})-
n\mu\overline\tau_{n,i}
\br)\,.
\end{multline}
Hence,
\begin{equation}
  \label{eq:218}
  \frac{1}{b_n\sqrt{n}}\,\sup_{s\in\R_+}(A_n(s)-\sum_{k=1}^nN'_k(s))=
  \sup_{i\in\N}(X_{n,i}+Y_{n,i}-
\frac{i}{n}
\,\frac{\beta_n}{\rho_n})\,,
\end{equation}
where
\begin{equation}
  \label{eq:97}
  X_{n,i}=
-\frac{1}{b_n\sqrt{n}}\,\sum_{k=1}^n\overline N'_k(\tau_{n,
i})
\end{equation}
and
\begin{equation}
  \label{eq:196}
    Y_{n,i}=
-\frac{1}{b_n\sqrt{n}}\,
n\mu\overline\tau_{n,i}\,.
\end{equation}
Clearly,
\begin{multline}
  \label{eq:116}
P(  \sup_{i\in\N}(X_{n,i}+Y_{n,i}-
\frac{i}{n}
\,\frac{\beta_n}{\rho_n})>r) \le 
P(  \sup_{i\in\N}(X_{n,i}-
\frac{i}{n}
\,\frac{\beta_n}{2\rho_n})>\frac{r}{2})
\\+P(  \sup_{i\in\N}(Y_{n,i}-
\frac{i}{n}
\,\frac{\beta_n}{2\rho_n})>\frac{r}{2})\,.
\end{multline}
The probabilities on the latter righthand side are worked on in
order.
By Petrov's device \cite{Pet87},
\[
    E\abs{\sum_{k=1}^n\overline
  N'_k(t)}^{ b_n^2}\le
(\frac{ b_n^2}{2}+1)^{ b_n^2}\,n
E\abs{\overline N_1'(t)}^{ b_n^2}
+ b_n^2(\frac{ b_n^2}{2}+1)^{ b_n^2/2}e^{ b_n^2/2+1}
 n^{ b_n^2/2}\bl(E\overline N_1'(t)^2\br)^{ b_n^2/2}\,.
\]
Hence,
\begin{multline}
  \label{eq:123}
    E\abs{\sum_{k=1}^n\overline  
  N'_k(\tau_{n,i})}^{ b_n^2}\le
(\frac{ b_n^2}{2}+1)^{ b_n^2}nE\abs{\overline N_1'(\tau_{n,i})}^{ b_n^2}
\\+ b_n^{2}(
\frac{ b_n^2}{2}
+1)^{ b_n^2/2}e^{ b_n^2/2+1}
 n^{ b_n^2/2}\bl(E\overline N_1'(\tau_{n,i})^2\br)^{ b_n^2/2}\,.
\end{multline}
By Lemma \ref{le:stat_moments}, for $p\ge2$\,,
\begin{equation}
  \label{eq:104}
    E\abs{\overline
  N_1'(t)}^p\le
\tilde C^p\bl((1+ \mu t)
    E(\mu \eta)^{p+1}
+(1+\mu t)^{p/2+1}p^{p/2+1}\bl(E(\mu\eta)^2\br)^{3p/2}\br)\,.
\end{equation}
It follows, with the use of Jensen's inequality, that, for some
$\tilde C'>0\,,$
\begin{equation}
  \label{eq:122}
    E\abs{\overline
  N_1'(\tau_{n,i})}^{ b_n^2}
\le
\tilde C'^{b_n^2}((1+E\xi_{n}i)E\eta^{ b_n^2+1}+b_n^{b_n^2}+
b_n^{ b_n^2}E\xi_n^{ b_n^2/2+1}
i^{ b_n^2/2+1})\,.
\end{equation}
Besides,
the bound in \eqref{eq:104} can be improved when $p=2$:
by Lemma \ref{le:stat_moments}, 
$     E\abs{\overline
  N_1'(t)}^2\le K' t\,,$
for some $K'>0$\,, so that
\begin{equation}
  \label{eq:5}
  E\overline N_1'(\tau_{n,i})^2\le K' E\xi_ni\,.
\end{equation}

Putting together
 \eqref{eq:97},\eqref{eq:123},   \eqref{eq:122} and \eqref{eq:5} yields
for some $\hat C>0$\,,
\begin{multline*}
    E|    X_{n,i}|^{b_n^2}\le
 \hat C^{b_n^2}
\bigl((\frac{b_n}{2}+\frac{1}{b_n})^{b_n^2}n^{1-b_n^2/2}((1+E(n\xi_n)\,\frac{i}{n})E\eta^{b_n^2+1}+b_n^{b_n^2}\\+b_n^{b_n^2}E(n\xi_n)^{b_n^2/2+1}\bl(\frac{i}{n}\br)^{b_n^2/2+1})
+b_n^2(\frac{1}{2}+
\frac{1}{b_n^2})^{b_n^2/2}
e^{b_n^2/2+1}(En\xi_n)^{b_n^2/2}\bl(\frac{i}{n}\br)^{b_n^2/2}\bigr)
\,.
\end{multline*}
Hence,
\[
  \norm{\max_{1\le i\le \lfloor nt\rfloor}\abs{X_{n,i}}}_n\le
\lfloor nt\rfloor^{1/b_n^2}\max_{1\le i\le \lfloor
  nt\rfloor}\norm{X_{n,i}}_n
\le \hat C \hat L_n(t)
\]
where
\begin{multline*}
    \hat L_n(t)=\hat C
 n^{1/b_n^2}t^{1/b_n^2}
\Bigl((\frac{b_n}{2}+\frac{1}{b_n})n^{1/b_n^2-1/2}\bl((1+E(n\xi_n)t)^{1/b_n^2}
(E\eta^{b_n^2+1})^{1/b_n^2}+b_n\\+b_n(E(n\xi_n)^{b_n^2/2+1})^{1/b_n^2}t^{1/2+1/b_n^2}\br)
+b_n^{2/b_n^2}(\frac{1}{2}+
\frac{1}{b_n^2})^{1/2}
e^{1/2+1/b_n^2}(En\xi_n)^{1/2}t^{1/2}\Bigr)
\,.
\end{multline*}
Thus,
for $t>0$ and $\alpha>0$\,,
\begin{equation}
  \label{eq:101}
    P(\max_{1\le i\le \lfloor  nt\rfloor}\abs{X_{n,i}}>
\alpha )^{1/b_n^2}\le\frac{1}{\alpha }\,\hat L_n(t)\,.
\end{equation}
In addition, the hypotheses imply that
 there exists $\breve C>0$ such that,
for $t\ge1$ and all $n$ great enough,
\begin{equation}
  \label{eq:215}
  \hat L_n(t)\le \breve C t^{2/3}\,.
\end{equation}
Now, using a device due to Prohorov \cite{Pro63}, see also Puhalskii
\cite{Puh99a} for an  application to moderate deviations of queues,
for $r>0$\,,
\begin{multline}
  \label{eq:87}
  P\bl(\sup_{i\in\N}(X_{n,i}-\frac{i}{n}\,\frac{\beta_n}{2\rho_n})>\frac{r}{2}\br)
\le P(\max_{1\le i\le \lfloor nr\rfloor}X_{n,i}>\frac{r}{2}
)\\
+\sum_{l=\lfloor \log_2(\lfloor nr\rfloor)\rfloor}^\infty P\bl(
\max_{2^l+1\le i\le  2^{l+1}}
(X_{n,i}-\frac{i}{n}\,\frac{\beta_n}{2\rho_n})>0)
\\
\le P(\max_{1\le i\le \lfloor nr\rfloor}X_{n,i}>\frac{r}{2})+
\sum_{l=\lfloor \log_2(\lfloor nr\rfloor)\rfloor}^\infty P\bl(
X_{n,2^l}-\frac{2^{l-1}}{n}\,\frac{\beta_n}{2\rho_n}>0)
\\
+\sum_{l=\lfloor \log_2(\lfloor nr\rfloor)\rfloor}^\infty P\bl(
\max_{2^l+1\le i\le  2^{l+1}}
(X_{n,i}-X_{n,2^l}-\frac{i-2^{l-1}}{n}\,\frac{\beta_n}{2\rho_n})>0)
\\
\le P(\max_{1\le i\le \lfloor nr\rfloor}X_{n,i}>\frac{r}{2})
+2\sum_{l=\lfloor \log_2(\lfloor nr\rfloor)\rfloor}^\infty P\bl(
\max_{1\le i\le
  2^{l}}X_{n,i}>\frac{2^{l-1}}{n}\,\frac{\beta_n}{2\rho_n})
\,.
\end{multline}
By \eqref{eq:101}, 
\begin{equation}
  \label{eq:207}
   P(\max_{1\le i\le\lfloor nr\rfloor}
X_{n, i}>\frac{r}{2}
)^{1/b_n^2}\le \frac{2}{r}
\,\hat L_n(r)\,.
\end{equation}
Similarly, with $l= \lfloor\log_2\lfloor
nr\rfloor\rfloor+m$\,, 
by \eqref{eq:101},
\begin{multline}
  \label{eq:103}
      P\bl(\max_{1\le i\le 2^l}
X_{n,i}>\frac{2^{l-1}}{n}\,
\frac{\beta_n}{2\rho_n}\br)^{1/b_n^2}\le
P\bl(\max_{1\le i\le\lfloor nr\rfloor 2^m}
X_{n,i}>2^{m-2}r
\frac{\beta_n}{4\rho_n}\br)^{1/b_n^2}\\\le
\frac{\hat L_n(r2^m)}{r2^{m}}\frac{16\rho_n}{\beta_n}\,.
\end{multline}

Since
$
  \sum_{m=0}^\infty2^{-m/3}<\infty\,,
$
\eqref{eq:215},  \eqref{eq:87}, \eqref{eq:207} and \eqref{eq:103} imply that,
 for some $C_X>0$ and all $r>0$\,,
\begin{equation}
  \label{eq:198}
    P\bl(\sup_{i\in\N}(X_{n,i}-\frac{i}{n}\,\frac{\beta_n}{2\rho_n})>\frac{r}{2}\br)^{1/b_n^2}\le \frac{C_X}{r^{1/3}}\,.
\end{equation}
A similar bound holds for $Y_{n,i}$\,.
In some more detail, by \eqref{eq:196} and Petrov \cite{Pet87}, 
\begin{multline*}
  E\abs{Y_{n,i}}^{b_n^2}\le
\mu^{b_n^2}
\bl((\frac{b_n}{2}+\frac{1}{b_n})^{b_n^2}n^{1-b_n^2/2}
E\abs{n\overline\tau_{n,1}}^{b_n^2}\frac{i}{n}
\\+b_n^2(\frac{1}{2}+\frac{1}{b_n^2})^{b_n^2/2}e^{b_n^2/2+1}
(E(n\overline \tau_{n,1})^{2})^{b_n^2/2}\bl(\frac{i}{n}\br)^{b_n^2/2}\br)\,.
\end{multline*}
Mimicking the derivation of \eqref{eq:198} obtains that, for some
$C_Y>0$ and all $r>0$\,,
\begin{equation}
  \label{eq:213}
      P\bl(\sup_{i\in\N}(Y_{n,i}-\frac{i}{n}\,\frac{\beta_n}{2\rho_n})>\frac{r}{2}\br)^{1/b_n^2}\le \frac{C_Y}{r^{1/3}}\,,
\end{equation}
By  \eqref{eq:218},
 \eqref{eq:116}, \eqref{eq:198}, and \eqref{eq:213},
 there 
exists $\overline C>0$ such that, for all $r>0$\,,
\[
\sup_n    P(\sup_{s\in\R_+}\frac{1}{b_n\sqrt{n}}(A_n(s)-\sum_{k=1}^nN'_k(s))>r)^{1/b_n^2}
\le \frac{\overline C}{r^{1/3}}\,.
\]
Recalling  \eqref{eq:246} yields, for $\theta>0$\,,
\begin{multline*}
  E(X_n'(t)^+)^{\theta b_n^2}=\int_0^\infty\theta b_n^2r^{\theta
    b_n^2-1}P(X_n'(t)>r)\,dr
\\
\le\theta b_n^2+\int_0^\infty 
\theta b_n^2(r+1)^{\theta b_n^2-1} P(\sup_{s\in\R_+}\frac{1}{b_n\sqrt{n}}
(A_n(s)-\sum_{k=1}^nN'_k(s))>r)\,dr\,,
\end{multline*}
which  implies \eqref{eq:15} for suitable $\theta$\,.

\end{proof}
Lemma \ref{le:negative part}
and Lemma \ref{le:norm_bound} imply
\eqref{eq:134}.  Theorem \ref{the:stat_exp_tight} has been proved.

\section{Exponential tightness of the stationary residual service times}
\label{sec:residual}  Attention is being turned to residual service
times. The hypotheses of 
Theorem \ref{the:inv} are assumed.
Suppose that initial conditions are chosen to render $(Q_n(t)\,,t\in\R_+)$\,,
supplemented with  the process of  forward recurrence
 times of the arrival process
 and with the process of residual service times, 
a stationary process, see p.348 in Asmussen \cite{Asm03}.
 Let
 $\chi_{n,i}(t)$ represent the residual service time of the
customer in the $i$th server at time $t$\,, with the understanding
that $\chi_{n,i}(t)=0$ if the $i$th server is idle  at $t$\,.
Let  $( S_n(x,t)\,,x\in\R_+)$ represent the complementary empirical
cdf of the residual service times of customers in service at $t$\,, i.e.,
 \begin{equation}
   \label{eq:10'''}
      S_n(x,t)=\frac{1}{n}\,\sum_{i=1}^n\ind{\chi_{n,i}(t)>x}\,.
 \end{equation}
Also, the following representation holds,
 \begin{equation}
   \label{eq:17}
S_n(x,t)=\frac{1}{n}\,   \sum_{i=1}^{ \hat A_n(t)}\ind{\eta_i+\hat\tau_{n,i}-t
>   x}\,.
 \end{equation}
The process $((S_n(x,t),x\in\R_+),t\in\R_+)$ is strictly stationary.
Let $(S_n(x)\,,x\in\R_+)$ represent an $\R_+^{\R_+}$--valued 
random variable distributed as $(S_n(x,t)\,,x\in\R_+)$\,.
\begin{theorem}
  \label{the:emp_cdf}
Under the hypotheses of Theorem \ref{the:inv},
the sequence of processes $ ( \sqrt{n}/b_n
( S_n(x)-(1-F^{(0)}(x)))\,, x\in \R_+)%
$ is $\bbC_0$--exponentially tight of order $b_n^2$ in
$\D(\R_+,\R)$\,. 
\end{theorem}
\begin{remark}\label{re:eq_lim}
  As a byproduct, the cdf of
$\chi_{n,i}$ converges to $F^{(0)}$\,, as $n\to\infty$\,.
\end{remark}
The groundwork is laid first.
 Both  $\hat A_n$ and $A_n$ 
are  processes with stationary increments
 and $E\hat
A_n(t)=EA_n(t)=\lambda_nt$\,. 
These processes are extended to  
processes with stationary increments on the
whole real line with the same finite--dimensional distributions so
that they may assume negative values on the negative halfline.
Introduce
\[
     F_n(x,t)=\frac{1}{n}\,
\sum_{i=1}^{\hat A_n(t)}(1-
F(x+t-\hat\tau_{n,i}))=
\int_0^{t}(1-
F(x+t-y))\,d\hat A_n(y)\,.
\]
Via a time
shift to the left by $t$\,,
$(F_n(x,t),x\in\R_+)$ is distributed as 
$\bl((1/n\,\int_{-t}^0(1-F(x-y))\,d\hat A_n(y)),x\in\R_+\br)$\,. 
Letting $t\to\infty$
implies  that,
 in distribution in $\D(\R_+,\R)$\,,
\[
(   F_n(x,t),x\in\R_+)\to
 (F_n(x),x\in\R_+)\,,
\]
where
\begin{equation}
  \label{eq:32'}
   F_n(x)=
\frac{1}{n}\, \int_{-\infty}^0(1-F(x-y))\,d\hat A_n(y)
\,.
\end{equation}
The latter quantity is finite as $EF_n(x)
=1/n\,\int_{-\infty}^0(1-F(x-y))\,\lambda_n\,dy$\,. 
It admits the interpretation as the proportion
of customers arrived at some time $y$ before time $0$ that 
remain past $x$\,.

As the distribution of $(S_n(x,t)\,, x\in\R_+)$ does not depend on $t$
and the distribution of $(F_n(x,t)\,, x\in\R_+)$ tends to a limit, as
$t\to\infty$\,, given $t_m\to\infty$\,, the sequence 
$\bl((S_n(x,t_m),(F_n(x,t_m)), x\in\R_+)\br)$ 
of random elements of $\D(\R_+,\R^2)$ is
tight, so, it converges along a subsequence to a limit. It may  be
assumed that the limit is $\bl((S_n(x),  F_n(x)), x\in\R_+\br)$\,.
Let \begin{multline}
  \label{eq:28}
    \tilde U_n(x,t)=\frac{\sqrt{n}}{b_n}\,( S_n(x,t)- F_n(x,t))
\\=\frac{1}{b_n\sqrt{n}}\,
\sum_{i=1}^{ \hat A_n(t)}(\ind{\eta_i+\hat\tau_{n,i}-t
>   x}-(1-F(x+t-\hat\tau_{n,i})))\,.
\end{multline}
Putting together what has been established,  along a sequence of
values of $t$\,,
in distribution in $\D(\R_+,\R)$\,,
\begin{equation}
  \label{eq:34}
  (\tilde U_n(x,t),x\in\R_+)
\to(\tilde U_n(x),x\in\R_+)\,,
\end{equation}
where
\[
\tilde U_n(x)=  \frac{\sqrt{n}}{b_n}\,
\bl( S_n(x)- F_n(x)\br)\,.
\]

The next heuristic makes it plausible that $ F_n(x) \to F^{(0)}(x)$, 
so, the appearance of 
  $F^{(0)}$  in Theorem \ref{the:emp_cdf} starts making sense.
Noting that, evidently, $\hat A_n(t)/n\to \mu t$ in probability, as
$n\to\infty$
 (cf.,
Lemma 3.2 in Puhalskii \cite{Puh23}) and 
 recalling \eqref{eq:73},
by integration by parts in \eqref{eq:32'}, 
  with the convergence below being
in probability,
\[
 F_n(x)=
\frac{1}{n}\,\int_{x}^\infty
\hat A_n(y-x)\,dF(y)\to \mu\int_x^\infty(y-x)\,dF(y)\\
=1-F^{(0)}(x)\,.
\]
For an actual proof of Theorem \ref{the:emp_cdf},
 two lemmas are in order.
Let, for $x\in\R_+$ and $t\in\R_+$\,,
\begin{equation}
  \label{eq:59}
  \Lambda_n(x,t)=
\frac{1}{n}\,
\sum_{i=1}^{\hat A_n(t)}
\int_{\eta_i\wedge(t-\hat\tau_{n,i})}^{\eta_i\wedge(t+x-\hat\tau_{n,i})}
\frac{dF(u)}{1-F(u)}
\end{equation}
and
\begin{equation}
  \label{eq:14'}
   \Lambda^\ast_n(t)=\sup_{x\in\R_+} \Lambda_n(x,t)=
\frac{1}{n}\,
\sum_{i=1}^{\hat A_n(t)}\,\int_{\eta_i\wedge(t-\hat\tau_{n,i})}^{\eta_i}
\frac{dF(u)}{1-F(u)}\,.
\end{equation}
\begin{lemma}
  \label{le:lambda}
Under the hypotheses of Theorem \ref{the:inv},
\begin{equation}
  \label{eq:27}
    \limsup_{n\to\infty}\sup_{t\in\R_+}\,\norm{\Lambda^\ast_n(t)}_{n,\theta}<\infty\,,
\end{equation}
where  $\theta>0$ is chosen for   \eqref{eq:134} to hold.
\end{lemma}
\begin{proof}
Since $  \{t-\hat\tau_{n,i}< u\}
=\{\hat A_n((t-u)^+)< i\}$\,,
by \eqref{eq:14'}, 
\begin{multline}
  \label{eq:92}
\norm{\Lambda^\ast_n(t)}_{n,\theta}=
\frac{1}{n}\,\norm{
\int_0^\infty
\sum_{i=1+ \hat A_n((t-u)^+ )}^{\hat A_n(t)}
\ind{u< \eta_i}\,\frac{dF(u)}{1-F(u)}}_{n,\theta}\\
\le \frac{1}{n}\,
\int_0^\infty
\norm{\sum_{i=1+ \hat A_n((t-u)^+ )}^{\hat A_n(t)}
\ind{u< \eta_i}}_{n,\theta}\,\frac{dF(u)}{1-F(u)}\,.
\end{multline}
Recall that, by \eqref{eq:40},
\begin{equation}
  \label{eq:8}
  {\hat A}_n(t)=(Q_n(0)-n)^++A_n(t)-(Q_n(t)-n)^+.
\end{equation}
One can see that
\begin{multline*}
  \sum_{i=1+\hat A_n((t-u)^+)
}^{\hat A_n(t) }
\ind{u< \eta_i}=
\sum_{i=1+\hat A_n((t-u)^+ )
}^{(Q_n(0)-n)^++A_n((t-u)^+)}\ind{u< \eta_i}\,\\
+
\sum_{i=1+(Q_n(0)-n)^++A_n((t-u)^+)}^{(Q_n(0)-n)^++A_n(t)}\ind{u< \eta_i}\,
-
\sum_{i=1+\hat A_n(t)}^{(Q_n(0)-n)^++A_n(t)}\ind{u< \eta_i}\,.
\end{multline*}
Hence,
\begin{multline}
  \label{eq:21}
\norm{
\sum_{i=1+\hat A_n((t-u)^+)
}^{\hat A_n(t) }
\ind{u< \eta_i}}_{n,\theta}
\le  \norm{
\sum_{i=1+\hat A_n((t-u)^+ )
}^{(Q_n(0)-n)^++A_n((t-u)^+)}\ind{u< \eta_i}\,
}_{n,\theta}
\\+\norm{
\sum_{i=1+\hat A_n(t)}^{(Q_n(0)-n)^++A_n(t)}\ind{u< \eta_i}
}_{n,\theta}+\norm{
\sum_{i=1+(Q_n(0)-n)^++A_n((t-u)^+)}^{(Q_n(0)-n)^++A_n(t)}\ind{u< \eta_i}\,
}_{n,\theta}
\,.
\end{multline}
In order to bound  the terms on the righthand side
of \eqref{eq:21}, the centred versions are looked at first.
 The summation in the first term    is
over customers that are in the queue at time $(t-u)^+$
so that the service times $\eta_i$  in the sum are independent
of $A_n((t-u)^+)$\,, $   Q_n((t-u)^+)$\,, and $Q_n(0)$\,.
Once again using Petrov's bound, Petrov \cite{Pet87}, as
well as \eqref{eq:8},
\begin{multline}
  \label{eq:85}
  E\abs{\sum_{i=1+\hat A_n((t-u)^+)
 }^{(Q_n(0)-n)^++A_n((t-u)^+)}
\bl(\ind{u< \eta_i}-(1-F(u))\br)}^{\theta b_n^2}
\\
=E\Bl(  E\abs{\sum_{i=1}^N\bl(
\ind{u< \eta_i}-(1-F(u))\br)}^{\theta b_n^2}\Big|_{N=(Q_n((t-u)^+)-n)^+}\Br)
\\
\le E\Bl((\frac{\theta b_n^2}{2}+1)^{\theta b_n^2}NE\abs{
\ind{u\le \eta_1}-(1-F(u))}^{\theta b_n^2}\Big|_{N=(Q_n((t-u)^+)-n)^+}\\
+\theta b_n^2(\frac{\theta b_n^2}{2}+1)^{\theta b_n^2/2}e^{\theta b_n^2/2+1}
\bl(NE(
\ind{u\le
  \eta_1}-(1-F(u)))^{2}\br)^{\theta b_n^2/2}\Big|_{N=(Q_n((t-u)^+)-n)^+}\Br)\\
=E\Bl((\frac{\theta b_n^2}{2}+1)^{\theta b_n^2}N\bl(F(u)(1-F(u))^{\theta b_n^2}+
(1-F(u))F(u)^{\theta b_n^2}\br)\Big|_{N=(Q_n((t-u)^+)-n)^+}\\
+\theta b_n^2(\frac{\theta b_n^2}{2}+1)^{\theta b_n^2/2}e^{\theta b_n^2/2+1}
\bl(N\bl(F(u)(1-F(u))\br)^{\theta b_n^2/2}\Big|_{N=(Q_n((t-u)^+)-n)^+}\Br)
\\\le (1-F(u))\Bl(\bl(\frac{\theta b_n^2}{2}+1\br)^{\theta b_n^2}
\norm{(Q_n((t-u)^+)-n)^+)}_{n,\theta}\\+
\theta b_n^2\bl(\frac{\theta b_n^2}{2}+1\br)^{\theta b_n^2/2}e^{\theta b_n^2/2+1}
\norm{(Q_n((t-u)^+)-n)^+}_{n,\theta}^{\theta b_n^2/2}
\Br)\,.\end{multline}
Assuming that $\theta\le 1$ implies that, for $n$ great,
\begin{multline*}
  \norm{\sum_{i=1+\hat A_n((t-u)^+)
}^{(Q_n(0)-n)^++A_n((t-u)^+) }\bl(
\ind{u< \eta_i}-(1-F(u))\br)}_{n,\theta}\\
\le(1-F(u))^{1/b_n^2}
\Bl(b_n^{1/(\theta b_n^2)+2}n^{1/(2\theta b_n^2)}
\norm{(X_n((t-u)^+))^+}_{n,\theta}^{1/(\theta b_n^2)}\\+
b_n^{2/(\theta b_n^2)+3/2}n^{1/4}e^{1/2+1/(\theta b_n^2)}\norm{X_n((t-u)^+)^+}_{n,\theta}^{1/2}
\Br)\,.
\end{multline*}
Noting that $\int_0^\infty (1-F(u))^{1/b_n^2-1}\,dF(u)=b_n ^2$
and keeping in mind \eqref{eq:8}, yield
\begin{multline}
  \label{eq:96}
    \frac{1}{n}\,
\int_0^\infty\norm{
\sum_{i=1+\hat A_n((t-u)^+ )
}^{(Q_n(0)-n)^++A_n((t-u)^+) }
\ind{u< \eta_i}}_{n,\theta}\,\frac{dF(u)}{{1-F(u)}}
\\
\le  \frac{1}{n}\,
\int_0^\infty\norm{
\sum_{i=1+\hat A_n(t-u)^+)}^{(Q_n(0)-n)^++A_n((t-u)^+) }
(\ind{u< \eta_i}-(1-F(u)))}_{n,\theta}\,\frac{dF(u)}{{1-F(u)}}\\
+\frac{1}{n}\,
\int_0^\infty\norm{
(Q_n((t-u)^+)-n)^+ }_{n,\theta}
\,dF(u)
\le\frac{ b_n^2}{n}\, \bl(b_n^{1/(\theta b_n^2)+2}\,n^{1/(2\theta b_n^2)}
\sup_{u\in\R_+}\norm{X_n(u)^+}_{n,\theta}^{1/(\theta b_n^2)}\\+
b_n^{2/(\theta b_n^2)+3/2}n^{1/4}e^{1/2+1/(\theta b_n^2)}
\sup_{u\in\R_+}\norm{X_n(u)^+}_{n,\theta}^{1/2}
\br)+\frac{b_n}{\sqrt{n}}\,\sup_{u\in\R_+}\norm{X_n(u)^+}_{n,\theta}
\,.
\end{multline}
Since $X_n(t)$ is a stationary process for the chosen initial
condition, $\sup_{u\in\R_+}\norm{X_n(u)^+}_{n,\theta}=\norm{(X_n^{(s)})^+}_{n,\theta}\,.$
It is bounded in $n$ by Lemma \ref{le:norm_bound}.
In addition, the hypotheses imply that
$b_n^4n^{1/(2\theta b_n^2)-1}\to0$\,, $b_n^{7/2}n^{-3/4}\to0$ and $b_n/\sqrt n\to0$\,.
Hence,
\begin{equation}
  \label{eq:16'}
  \lim_{n\to\infty}\frac{1}{n}\,
\sup_t\int_0^\infty\norm{
\sum_{i=1+\hat A_n((t-u)^+) }
^{(Q_n(0)-n)^++A_n((t-u)^+) }
\ind{u< \eta_i}}_{n,\theta}\,\frac{dF(u)}{{1-F(u)}}=0\,.
\end{equation}
The second and third terms on the righthand side of \eqref{eq:21} are
handled in a like fashion.
Similarly to \eqref{eq:85},
\begin{multline*}
    E\abs{\sum_{i=1+\hat A_n(t )
}^{(Q_n(0)-n)^++A_n(t) }(
\ind{u< \eta_i}-(1-F(u))}^{\theta b_n^2}
\le (1-F(u))\Bl(\bl(\frac{\theta b_n^2}{2}+1\br)^{\theta b_n^2}\norm{(Q_n(t)-n)^+}_{n,\theta}\\+
\theta b_n^2\bl(\frac{\theta b_n^2}{2}+1\br)^{\theta b_n^2/2}
e^{\theta b_n^2/2+1}\norm{(Q_n(t)-n)^+}_{n,\theta}^{\theta b_n^2/2}
\Br)\,,
\end{multline*}
which implies, in analogy with \eqref{eq:96} and \eqref{eq:16'}, that
\begin{equation}
  \label{eq:12}
    \lim_{n\to\infty}\frac{1}{n}\,\sup_t
\int_0^\infty\norm{
\sum_{i=1+\hat A_n(t)
}^{(Q_n(0)-n)^++A_n(t) }
\ind{u< \eta_i}}_{n,\theta}\,\frac{dF(u)}{{1-F(u)}}=0\,.
\end{equation}
By $A_n$ being an equilibrium process, as well as by
  arrivals and service being independent,
\begin{multline*}
      E\abs{\sum_{i=1+(Q_n(0)-n)^++A_n((t-u)^+)}^{(Q_n(0)-n)^++A_n(t) }\bl(
\ind{u< \eta_i}-(1-F(u))\br)}^{\theta b_n^2}\\=
E\abs{\sum_{i=1}^{A_n(t-(t-u)^+) }\bl(
\ind{u< \eta_i}-(1-F(u))\br)}^{\theta b_n^2}
\\
\le(1-F( u))\Bl(\bl(\frac{\theta b_n^2}{2}+1\br)^{\theta b_n^2}\norm{
A_n(t\wedge u)}_{n,\theta}+
\theta b_n^2\bl(\frac{\theta b_n^2}{2}+1\br)^{\theta b_n^2/2}e^{\theta
  b_n^2/2+1}\norm{A_n(t\wedge u)}_{n,\theta}^{\theta b_n^2/2}
\Br)\,
\end{multline*}
so that
\begin{multline*}
\sup_t  \norm{
\sum_{i=1+(Q_n(0)-n)^++A_n((t-u)^+)}^{(Q_n(0)-n)^++A_n(t)}\bl(
\ind{u< \eta_i}-(1-F(u))\br)}_{n,\theta}\\
\le (1-F(u))^{1/b_n^2}\bl((\frac{\theta b_n^2}{2}+1)\norm{
A_n(u)}_{n,\theta}^{1/(\theta b_n^2)}\\+
b_n^{2/b_n^2}\bl(\frac{\theta b_n^2}{2}+1\br)^{1/2}e^{1/2+1/b_n^2}\norm{
A_n(u)}_{n,\theta}^{1/2}\br)\,.
\end{multline*}
According to \eqref{eq:130}, recalling that $\theta\le1$\,,
$\norm{A_n(u)}_{n,\theta}\le\norm{A_n(u)}_{n}\le  Cn(1+u)$\,.
Note that
$u^{1/(\theta b_n^2)}(1-F(u))^{1/b_n^2}\le u^{1/(\theta
  b_n^2)-1}(E\eta^{b_n^2})^{1/b_n^2}\to0$\,, as $u\to\infty$\,. Therefore,
\begin{multline*}
  b_n^2\int_1^\infty(1-F(u))^{1/b_n^2-1} u^{1/(\theta b_n^2)}\,dF(u)=
-(1-F(1))^{1/b_n^2}\\+\int_1^\infty(1-F(u))^{1/b_n^2}(\theta b_n^2)^{-1}u^{1/(\theta
  b_n^2)-1}\,du
\le \frac{\norm{\eta}_n}{
\theta b_n^2}\int_1^\infty u^{1/(\theta b_n^2)-2}\,du=
\frac{\norm{\eta}_n}{
\theta b_n^2-1}\,.
\end{multline*}
 Similarly,
 \begin{multline*}
    b_n^2\int_1^\infty(1-F(u))^{1/b_n^2-1} u^{1/2}\,dF(u)
=-(1-F(1))^{2/b_n^2}\\+
\frac{1}{2}\int_1^\infty(1-F(u))^{1/b_n^2}u^{-1/2}\,du
\le \norm{\eta}_n\int_1^\infty u^{-3/2}\,du=
2\norm{\eta}_n\,.
 \end{multline*}
It follows that
\begin{multline*}
      \lim_{n\to\infty}\sup_{t}\frac{1}{n}\,
\int_0^\infty\norm{
\sum_{i=1+(Q_n(0)-n)^++A_n((t-u)^+)}^{(Q_n(0)-n)^++A_n(t) }\bl(
\ind{u< \eta_i}-(1-F(u))\br)}_{n,\theta}\,\frac{dF(u)}{{1-F(u)}}\\=0\,.
\end{multline*}
Also
\[      \limsup_{n\to\infty}\sup_{t}\frac{1}{n}\,
\int_0^\infty\norm{A_n(t)-A_n((t-u)^+)}_{n,\theta}\,dF(u)
<\infty
\]
so that
\begin{equation}
  \label{eq:10}
      \limsup_{n\to\infty}\frac{1}{n}\,\sup_t
\int_0^\infty\norm{
\sum_{i=1+(Q_n(0)-n)^++A_n((t-u)^+)}^{(Q_n(0)-n)^++A_n(t) }
\ind{u< \eta_i}}_{n,\theta}\,\frac{dF(u)}{{1-F(u)}}<\infty\,.\end{equation}
Putting together 
 \eqref{eq:92}, \eqref{eq:21},  \eqref{eq:16'}, \eqref{eq:12}, and  \eqref{eq:10}
yields  \eqref{eq:27}.
\end{proof}
Let $\mathcal{F}_{x,t}$ represent
 the complete $\sigma$--algebra generated by
the random variables $\hat A_n(s)$\,, where $s\le t$\,, and by
$\ind{\hat\tau_{n,i}\le u}\ind{\eta_i\le y}$\,, where
 $i\in\N$ and $u+y\le x+ t$\,.
One can see that $\mathbf F_t=
(\mathcal{F}_{x,t}\,,x\in\R_+)$ is a filtration, i.e., it is a
rightcontinuous flow of complete $\sigma$--algebras.
\begin{lemma}
  \label{le:mart}
Given $t>0$\,, the process $( M_n(x,t)\,, x\in\R_+)$ such that
\begin{equation}
  \label{eq:42}
   M_n(x,t)=\frac{1}{b_n\sqrt{n}}\,\sum_{i=1}^{\hat A_n(t)}
\bl(\ind{t<\hat\tau_{n,i}+\eta_i\le x+t}
-\int_{\eta_i\wedge(t-\hat\tau_{n,i})}^{\eta_i\wedge(x+t-\hat\tau_{n,i})}
\frac{dF(u)}{1-F(u)}\br)
\end{equation}
is a martingale with respect to $\mathbf F_t$\,. 
\end{lemma}
\begin{proof}
It is required to prove that if $x\le y$\,, then
\begin{equation}
  \label{eq:146}
  E( M_n(y,t)|\mathcal{F}_{x,t})= M_n(x,t)\,.
\end{equation}
The proof draws on  Lemma 3.1
in Puhalskii and Reed \cite{PuhRee10}.
Let
  \[
    L_n(x,s)=\frac{1}{b_n\sqrt{n}}\,\sum_{i=1}^{\hat A^n(s)}
\Bl(\ind{\eta_i\le x}-\int_0^{\eta_i\wedge x}
\frac{dF(u)}{1-F(u)}\Br)\,.
  \]
Then,
\begin{equation}
  \label{eq:141}
   M_n(x,t)=\int_{\R_+^2}\ind{s\le t}\ind{u+s\le x+t}dL_n(u,s)\,.
\end{equation}
It can be proved, in analogy with the proof of
 Lemma B.2 in Puhalskii and Reed \cite{PuhRee10}, that
for  $x\le x'$ and $s\le s'$\,,
\begin{equation}
  \label{eq:142}
  \mathbf{E}\bl(\square L_n((x,s),(x',s'))|\mathcal{F}_{x,s}\br)=0\,,
\end{equation}
where 
$\square
L_n\bl((x,s),(x',s')\br)=
L_n(x',s')-L_n(x',s)-L_n(x,s')+L_n(x,s)$\,.
Let $R_{x,t}$ represent the integration region in \eqref{eq:141}:
$R_{x,t}=\{(u,s):\, 
0<u\le x+t-s,0<s\le t\}$\,.
It  is   approximated with a union of rectangles.
More specifically, let
$0=s^{(l)}_0<s^{(l)}_1<\ldots<s^{(l)}_{l-1}<s^{(l)}_l=t$\, be a
partition of $[0,t]$\,, where $l\in\N$\,. It is assumed that
\begin{equation}
  \label{eq:152}
  \lim_{l\to\infty}
\max_{0\le i\le l-1}(s^{(l)}_{i+1}-s^{(l)}_i)= 0\,.
\end{equation}
The region $ R_{x,t}^{(l,i)}=\{(u,s):\, 
u\in(0, x+t-s]\,, s\in(s^{l}_i,s^{(l)}_{i+1}]\}$
is approximated with the region
$(0,x+t-s^{(l)}_{i}]\times (s^{l}_i,s^{(l)}_{i+1}]$\,.
Accordingly, $R_{x,t}$ is approximated with
$\cup_{i=1}^l R_{x,t}^{(l,i)}$ and $   M_n(x,t)$
is approximated with
\begin{multline}
  \label{eq:147}
     M_n^{(l)}(x,t)=\int_{\R_+^2}
\ind{(u,s)\in\cup_{i=1}^l R_{x,t}^{(l,i)}}\,dL_n(u,s)\\=
\sum_{i=0}^{l-1}\square L_n((0,s^{(l)}_i),(x+t-s^{(l)}_{i},s^{(l)}_{i+1}))\,.
\end{multline}
Similarly,
$ M_n(y,t)$ is approximated with
\begin{equation}
  \label{eq:148}
  M_n^{(l)}(y,t)=
\sum_{i=0}^{l-1}\square L_n((0,s^{(l)}_i),(y+t-s^{(l)}_{i},s^{(l)}_{i+1}))\,.
\end{equation}
By \eqref{eq:142},
\begin{multline*}
  E(\square L_n((0,s^{(l)}_i),(y+t-s^{(l)}_{i},s^{(l)}_{i+1}))
|\mathcal{F}_{x+t-s^{(l)}_{i},s^{(l)}_{i}})
\\=E(\square L_n((0,s^{(l)}_i),(x+t-s^{(l)}_{i},s^{(l)}_{i+1}))
|\mathcal{F}_{x+t-s^{(l)}_{i},s^{(l)}_{i}})\,.
\end{multline*}
Since $\mathcal{F}_{x,t}\subset 
\mathcal{F}_{x+t-s^{(l)}_{i},s^{(l)}_{i}}$\,,
\[
  E(\square L_n((0,s^{(l)}_i),(y+t-s^{(l)}_{i},s^{(l)}_{i+1}))
|\mathcal{F}_{x,t})
=E(\square L_n((0,s^{(l)}_i),(x+t-s^{(l)}_{i},s^{(l)}_{i+1}))
|\mathcal{F}_{x,t})
\]
and by \eqref{eq:147} and \eqref{eq:148},
\[
    E( M_n^{(l)}(y,t)|\mathcal{F}_{x,t})=E(
 M_n^{(l)}(x,t)|\mathcal{F}_{x,t})\,.
\]
By \eqref{eq:141}, \eqref{eq:147}, \eqref{eq:148} and
\eqref{eq:152}, as $l\to\infty$\,, $ M_n^{(l)}(y,t)\to
M_n(y,t)$
and $ M_n^{(l)}(x,t)\to M_n(x,t)$\,.
Similarly to the proof of (C.20a) and (C.20b) in Puhalskii and Reed
\cite{PuhRee10}, 
\begin{align*}
  \lim_{l\to\infty}E( M_n^{(l)}(y,t)|\mathcal{F}_{x,t})=
E(  M_n(y,t)|\mathcal{F}_{x,t})
\intertext{ and}
\lim_{l\to\infty}E( M_n^{(l)}(x,t)|\mathcal{F}_{x,t})=
  M_n(x,t)\,,
\end{align*}
proving \eqref{eq:146}.

\end{proof}
 
\begin{proof}[Proof of Theorem \ref{the:emp_cdf}]

The measure of jumps of $( M_n(x,t)\,, x\in\R_+)$
 is
\[
  \tilde\mu_n([0,x],\Gamma,t)=\ind{1/(b_n\sqrt{n})\in\Gamma}
\sum_{i=1}^{\hat A_n(t)}
\ind{t<\hat\tau_{n,i}+\eta_i\le t+ x}\,,
\]
where $\Gamma\subset \R\setminus\{0\}$\,.
The associated predictable measure of jumps is
\[
    \nu_n([0,x],\Gamma,t)=\ind{1/(b_n\sqrt{n})\in\Gamma}
\sum_{i=1}^{\hat A_n(t)}\int_{\eta_i\wedge(t-\hat\tau_{n,i})}^{\eta_i\wedge(x+t-\hat\tau_{n,i})}
\frac{dF(u)}{1-F(u)}\,.
\]
Define the stochastic cumulant
\begin{multline*}
   G_n(\alpha,x,t)=\int_0^x\int_\R
(e^{\alpha v}-1-\alpha v)\nu_n(dy,dv,t)\\=
(e^{\alpha /(b_n\sqrt{n})}-1-\frac{\alpha}{ b_n\sqrt{n}})\sum_{i=1}^{\hat A_n(t)}\int_{\eta_i\wedge(t-\hat\tau_{n,i})}^{\eta_i\wedge(x+t-\hat\tau_{n,i})}
\frac{dF(u)}{1-F(u)}\,.
\end{multline*}
On recalling \eqref{eq:59}, 
\[
    G_n(\alpha,x,t)=
(e^{\alpha /(b_n\sqrt{n})}-1-\frac{\alpha}{ b_n\sqrt{n}})
n\Lambda_n(x,t)\,.
\]
By \eqref{eq:14'},
\[
  \sup_{x\in\R_+}G_n(\alpha,x,t)\le
(e^{\alpha /(b_n\sqrt{n})}-1-\frac{\alpha}{ b_n\sqrt{n}})
n\Lambda_n^\ast(t)\,.
\]
The process $(e^{\alpha  M_n(x,t)-
G_n(\alpha,x,t)}\,,x\in\R_+)$ is a local martingale, so,
\[
  Ee^{\alpha  M_n(\sigma,t)-
G_n(\alpha,\sigma,t)}\le1\,,
\]
for arbitrary stopping time $\sigma$\,.  Invoking Lemma 3.2.6 on p.282
in Puhalskii  \cite{Puh01} yields, for 
 $\alpha>0$\,,  $\gamma>0$\,,  and $r>0$\,, 
\[
  P(\sup_{y\le x} M_n(y,t)\ge r)\le e^{\alpha  b_n^2(\gamma-r)}+
P(G_n(\alpha  b_n^2,x,t)\ge\alpha  b_n^2\gamma)
\]
so that
\[
    P(\sup_{x \in\R_+} M_n(x,t)\ge r)\le
e^{\alpha  b_n^2(\gamma-r)}+
P(      G_n^\ast(\alpha,t)
\ge\alpha  b_n^2\gamma)\,,
\]
where
\[
      G_n^\ast(\alpha,t)=
(e^{\alpha /(b_n\sqrt{n})}-1-\frac{\alpha}{ b_n\sqrt{n}})
n\Lambda_n^\ast(t)\,.
\]

Thus, 
\[  \limsup_{n\to\infty}\limsup_{t\to\infty}
  P(\sup_{x\in\R_+} M_n(x,t)\ge r)^{1/b_n^2}
\le e^{\alpha (\gamma-r)}+
\limsup_{n\to\infty}\limsup_{t\to\infty}
\frac{\norm{G_n^\ast(\alpha,t)}_{n,\theta}}{
\alpha b_n^2\gamma}\,,
\]
which implies, in light of Lemma \ref{le:lambda}, that
\[
    \lim_{r\to\infty}\limsup_{n\to\infty}\limsup_{t\to\infty}
  P(\sup_{ x\in\R_+} M_n(x,t)\ge r)^{1/b_n^2}=0\,.
\]
By a similar line of reasoning, 
\[
  \lim_{r\to\infty}\limsup_{n\to\infty}\limsup_{t\to\infty}
  P(\sup_{x\in\R_+}(- M_n(x,t))\ge r)^{1/b_n^2}=0
\]
so that
\begin{equation}
  \label{eq:75}
  \lim_{r\to\infty}\limsup_{n\to\infty}\limsup_{t\to\infty}
  P(\sup_{x\in\R_+}\abs{ M_n(x,t)}\ge r)^{1/b_n^2}=0\,.
\end{equation}
 Similarly, it is proved that, for $\epsilon>0$ and $a>0$\,,
\begin{equation}
  \label{eq:52}
  \lim_{\delta\to0}
  \limsup_{n\to\infty}\limsup_{t\to\infty}\sup_{y\le a}
  P(\sup_{x\in[y,y+\delta]}
\abs{ M_n(x,t)- M_n
(y,t)}\ge \epsilon)^{1/b_n^2}=0\,.
\end{equation}
(The inequality in \eqref{eq:23'} below also works.)
It is proved next that
\begin{equation}
  \label{eq:7'}
  \lim_{x\to\infty}\limsup_{n\to\infty}
\limsup_{t\to\infty}  P(\sup_{y\ge x}
\abs{ M_n(y,t)- M_n(x,t)}\ge \epsilon)^{1/b_n^2}=0\,.
\end{equation}
Since, with $\alpha\in\R$\,, the process $(e^{\alpha
(  M_n(x+y,t)- M_n(x,t))-
(G_n(\alpha,x+y,t)-G_n(\alpha,x,t))}\,,y\in\R_+)$ is a local martingale so that,
for arbitrary stopping time
$\sigma$\,,
\[
E e^{\alpha
(   M_n(\sigma+y,t)- M_n(y,t))-(
G_n(\alpha,\sigma+y,t)-G_n(\alpha,\sigma,t))}\le1\,,\]  by
Lemma 3.2.6 on p.282 in Puhalskii  \cite{Puh01},
for arbitrary $\alpha>0$\,,  $\gamma>0$\,, $\epsilon>0$ and $\delta>0$\,, 
\begin{multline}
  \label{eq:23'}
      P(\sup_{y\in[x,x+\delta]}(  M_n(y,t)- M_n(x,t))\ge \epsilon)^{1/b_n^2}\le
e^{\alpha(\gamma- \epsilon)}
\\+P(G_n(\alpha b_n^2,x+\delta,t)-G_n(\alpha b_n^2,x,t)\ge\alpha
  b_n^2\gamma)^{1/b_n^2}\,.
\end{multline}
On letting $\delta\to\infty$\,,
\begin{multline}
  \label{eq:9}
    P(\sup_{y\ge x}(  M_n(y,t)- M_n(x,t))\ge \epsilon)^{1/b_n^2}\le
e^{\alpha(\gamma- \epsilon)}\\
+P(G_n^\ast(\alpha b_n^2,t)-G_n(\alpha b_n^2,x,t)\ge\alpha
  b_n^2\gamma)^{1/b_n^2}\,.
\end{multline}
Since $ \Lambda_n^\ast(t)-\Lambda_n(x,t)\downarrow 0$\,, as
$x\to\infty$\,, and 
$\norm{\Lambda_n^\ast(t)-\Lambda_n(x,t)}_{n,\theta}\le \norm{\Lambda_n^\ast(t)}_{n,\theta}$\,,
by dominated convergence and Lemma \ref{le:lambda}, 
\[
\lim_{x\to\infty}\limsup_{n\to\infty}\limsup_{t\to\infty}
  P(G_n^\ast(\alpha b_n^2,t)-G_n(\alpha b_n^2,x,t)\ge\alpha
  b_n^2\gamma)^{1/b_n^2}=0\,.
\]
Assuming that   $\epsilon>\gamma$ in \eqref{eq:9} and letting
$\alpha\to\infty$  yield \eqref{eq:7'}.

Owing to \eqref{eq:28} and \eqref{eq:42},
\beq{sm'}
\tilde U_n(x,t)=-\int_0^x
\frac{\tilde U_n(y,t)}{1-F(y)}\,dF(y) + M_n(x,t)\,,\;\;
x\in\R_+,\,t\in\R_+\,.
\eeq
By \eqref{eq:75}, recalling that $F(y)<1$ for all $y$\,, yields, for
$x>0$\,, via the Gronwall--Bellman inequality,
\begin{equation}
    \label{eq:66a}
\lim_{r\to\infty}  \limsup_{n\to\infty}\limsup_{t\to\infty}
  P(\sup_{y\le x}\abs{\tilde U_n(y,t)}\ge r)^{1/b_n^2}=0
\end{equation}
and, for  $a>0$\,,
\begin{equation}
  \label{eq:52a}
  \lim_{\delta\to0}  \limsup_{n\to\infty}\limsup_{t\to\infty}
\sup_{y\in[0,a]}  P(\sup_{x\in[y,y+\delta]}
\abs{\tilde U_n(x,t)-\tilde U_n
(y,t)}\ge \eta)^{1/b_n^2}=0\,.
\end{equation}
Thanks to  \eqref{eq:66a}, \eqref{eq:52a}, \eqref{eq:34},
Fatou's lemma and Theorem
A.3 in Puhalskii \cite{Puh23}, 
the sequence
$(\tilde U_n(x),x\in\R_+)$ is 
$\bbC$--exponentially tight in $\D(\R_+,\R)$\,.

  In  addition, by 
\eqref{eq:30}, \eqref{eq:73}, \eqref{eq:46}, 
\eqref{eq:32'}, and \eqref{eq:8},
  \begin{multline}
  \label{eq:102}
  \frac{\sqrt{n}}{b_n}\,( F_n(x)-(1-F^{(0)}(x)))=
\int_{0}^\infty 
\frac{\sqrt{n}}{b_n}\,\bl(\frac{1}{n}\hat
A_n(s)-\mu s\br)\,
\,d_sF(x+s)\\
=\int_{x}^\infty
Y_n (s-x)
\,dF(s)  
-\int_{x}^\infty
X_n(s-x)\,dF(s) 
\,.
\end{multline}
By the LD convergence of $Y_n$ to $Y$ and the continuous mapping
principle, for $L>x$\,,
$\int_{x}^L
Y_n(s-x)
\,dF(s)  $  LD converges to
$\int_{x}^L
Y(s-x)
\,dF(s)  $\,. Also, the inequality
\eqref{eq:128}, the definitions  \eqref{eq:46} and
\eqref{eq:16}, and the heavy traffic condition \eqref{eq:105}  imply that
\[
  \lim_{L\to\infty}\limsup_{n\to\infty}\int_L^\infty \norm{Y_n(s-x)}_n\,dF(s)
=0\,.
\]
It follows that
$  \int_{x}^\infty
Y_n(s-x)
\,dF(s)$ LD converges to 
$  \int_{x}^\infty
Y(s-x)
\,dF(s)$\,.
Furthermore, by a similar line of reasoning and the joint convergence
in the statement of Theorem \ref{the:heavy_traffic},
$  \int_{x}^\infty
Y_n(s-x)
\,dF(s)$
and $  \int_{x}^\infty
X_n(s-x)
\,dF(s)$ jointly LD converge to 
$  \int_{x}^\infty
Y(s-x)
\,dF(s)$ 
and $  \int_{x}^\infty
X(s-x)
\,dF(s)$ (recall  the fact that $X_n$ is stationary and
Remark \ref{re:stat_exp}) so that
 the leftmost side of
\eqref{eq:102} LD converges to
$\int_{x}^\infty
Y(s-x)
\,dF(s)  
-\int_{x}^\infty
X(s-x)\,dF(s) 
\,.$

Since 
\begin{equation}
  \label{eq:121}
  \frac{\sqrt{n}}{b_n}\,( S_n(x)-(1-F^{(0)}(x)))
=\tilde U_n(x)
+\frac{\sqrt{n}}{b_n}( F_n(x)-(1-F^{(0)}(x)))\,,
\end{equation}
the sequence   $(\sqrt{n}/b_n)\,( S_n(x)-(1-F^{(0)}(x)))$ is
$\bbC$--exponentially tight\,.

In order to show that the LD limit points of the sum on the righthand side
of \eqref{eq:121} tend to zero as
$x\to\infty$\,, denote by  $\tilde\Pi$  a deviability on
$\bbC(\R_+,\R^3)$ that is a subsequential LD limit of
the distributions of $(X_n,Y_n,\tilde U_n)$
and let $(X,Y,\tilde U)$ denote the canonical idempotent variable
on $\bbC(\R_+,\R^3)$\,. 
Since by Theorem \ref{the:inv}
 $W(t)=(   Y(t)+\beta\mu t)/\sigma$ defines
 a standard Wiener idempotent process,
\[
\int_{x}^\infty
Y  (s-x)\,dF(s)=
\int_{x}^\infty
(\sigma W(s-x)-\beta\mu(s-x))\,dF(s)\,.
\]
By integration by parts,
\[
  \int_{x}^\infty
 W(s-x)\,dF(s)
=\int_{0}^\infty
 (1-F(x+s))\dot W(s)\,ds\,.
\]
By the Cauchy--Schwartz inequality and by 
$\int_0^\infty\dot W(s)^2\,ds$ being finite
$\tilde\Pi$--a.e., the  quantity on the latter righthand side converges to
zero, as $x\to\infty\,$, $\tilde\Pi$-a.e. Also, 
$\int_x^\infty(s-x)\,dF(s)\to0$\,, as $x\to\infty$\,.
Thus, $\int_{x}^\infty
Y(s-x)
\,dF(s)  \to 0$\,, as $x\to\infty$\,, $\tilde\Pi$-a.e.
By the LD convergence of $X_n$ to $X$\,,
\[
  \liminf_{n\to\infty}\sup_t\norm{X_{n}(t)}_{n,\theta}\ge \sup_{t}
\sup_{(X,Y,\tilde U)\in\bbC(\R_+,\R^3)} X(t)\tilde\Pi(X,Y,\tilde U)^{1/\theta}\,.
\]
The lefthand side being finite due to Theorem
\ref{the:stat_exp_tight},
 $\sup_{t\in\R_+}X(t)<\infty$ $\tilde \Pi$--a.e., which implies that
$\int_{x}^\infty
X(s-x)\,dF(s) \to0$\,, as $x\to\infty$\,, $\tilde\Pi$--a.e.
This concludes the proof of the claim that the  LD limit points of
$(\sqrt{n}/b_n)\,( F_n(x)-(1-F^{(0)}(x)))$\, tend to zero, as $x\to\infty$\,.

In order to show that the LD limit points of $\tilde U_n(x)$ tend to 0, as
$x\to\infty$\,, note that, by   \eqref{eq:34} and \eqref{sm'}, there exists the
 limit in distribution in $\D(\R_+,\R)$ $ M_n(x)=
\lim_{t\to\infty} M_n(x,t)$ and
\[
  \tilde U_n(x)=-\int_0^x
\frac{\tilde U_n(y)}{1-F(y)}\,dF(y)+ M_n(x)\,,\;\;
x\in\R_+\,.
\]
Taking an LD subsequential 
limit in the latter equation yields, by the continuous
mapping principle, 
\begin{equation}
  \label{eq:21'}
  \tilde U(x)=-\int_0^x
\frac{\tilde U(y)}{1-F(y)}\,dF(y)+ M(x)\,,\;\;
x\in\R_+\,,  \tilde\Pi\text{-a.e.}\,,
\end{equation}
where
$ M=( M(x),x\in\R_+)$
 is a continuous idempotent process.
Solving \eqref{eq:21'} yields
\begin{multline}
  \label{eq:22'}
  \tilde U(x)=(F(x)-1)\int_0^x
\frac{M(y)}{(1-F(y))^2}\,dF(y)+M(x)
=(F(x)-1)\int_0^x
\frac{M(y)-M(x)}{(1-F(y))^2}\,dF(y)\\
+(1-F(x))M(x)
\,.
\end{multline}
Since $M$ is an LD limit point of $ M_n$\,,
the limit \eqref{eq:7'} implies that
\[
    \lim_{x\to\infty} \tilde\Pi((X,Y,\tilde U):\,\sup_{y\ge x}
\abs{M(y)-M(x)}\ge \epsilon)=0
\]
so that, given $a>0$\,, $M(x)$ converges to a finite limit as
$x\to\infty$ uniformly on the set $\{(X,Y,\tilde U):\,\tilde\Pi(X,Y,\tilde U)\ge a\}$\,, cf. the proof of Lemma \ref{the:c0}.
Since  $F(x)\to1$, as $x\to\infty$\,,
by \eqref{eq:22'}, 
$  \tilde U(x)\to0$ uniformly on that set.
 \end{proof}

\section{Coupling of idempotent processes and
the proof of the main result}
\label{sec:moder-devi-stat}
 This section  proves Theorem \ref{the:inv},
so, the hypotheses  of the theorem are assumed throughout.
Let  $\Pi^W$ 
 and $\Pi^K$ denote the idempotent distributions of 
the  standard Wiener idempotent
process
$W$
and of the   Kiefer idempotent process $K$\,, respectively. Let
 $ \Pi=\Pi^W\times \Pi^K$
 denote the product
deviability, which is an idempotent measure on 
 $ \Upsilon=
\bbC(\R_+,\R)\times \bbC(\R_+,\bbC([0,1],\R))\,.$
  Given $x\in\R$ and $\phi\in\bbC_0(\R_+,\R)$\,,
let $\Pi_{x,\phi}$ denote the idempotent distribution of the solution
$X=(X(t)\,,t\in\R_+)$
to the equation 
\begin{multline*}
          X(t)=(1-F(t))x^+ -(1-F^{(0)}(t))x^-+\phi(t)-\beta F^{(0)}(t)
+\int_0^tX(t-s)^+\,dF(s)
\\
+\sigma\int_0^t\bl(1-F(t-s)\br)\,\dot W(s)\,ds
+\int_0^t\int_{0}^t \ind{x'+s\le t}\,\dot K(F(x'),\mu s)\,dF(x')\,\mu \,ds
\,,
\end{multline*}
where the standard Wiener idempotent process $W$ and the Kiefer idempotent
process $K$ are independent.
One can see that the mapping $(W,K)\to X$ is strictly Luzin so
that $\Pi_{x,\phi}$ is a deviability. 
Sample paths of $X$ are denoted by
$q=(q(t)\,,t\in\R_+)\in\bbC(\R_+,\R)$\,, i.e., $q(t)=X(\upsilon,t)$\,, where $\upsilon\in\Upsilon$\,.
Clearly, $\Pi_{x,\phi}(q)=\Pi(X=q)$\,.
Let also, for  $y\in \R$\,,
\begin{equation}
  \label{eq:225}
  \Pi_{x,\phi,t}(y)=\Pi_{x,\phi}(q(t)=y)(=\sup_{q:\,q(t)=y}\Pi_{x,\phi}(q))\,.
\end{equation}
(Since the set $\{q:\,q(t)=y\}$ is closed and $\Pi_{x,\phi}(q)$ is upper
compact in $q$\,, the supremum in \eqref{eq:225} is attained.)
It is noteworthy that $\Pi_{x,\phi,t}(y)$ is a deviability density on $\R$
as a function of $y$ because the map $q\to q(t)$ is continuous. 

Let the idempotent process 
$\hat X= (\hat X(t)\,,t\in\R_+)$ be defined
as the special case of $X$ where $x=-\beta$ and $\phi(t)=0$\,, i.e.,
\begin{multline}
  \label{eq:137}
            \hat{X}(t)=-\beta
+\int_0^t\hat X(t-s)^+\,dF(s)
+\sigma\int_0^t\bl(1-F(t-s)\br)\,\dot W(s)\,ds\\
+\int_0^t\int_{0}^t \ind{x+s\le t}\,\dot K(F(x),\mu s)\,dF(x)\,\mu \,ds\,.
\end{multline}
Let 
$\hat \Pi_t(y)=\Pi(\hat X(t)=y)=\Pi_{-\beta,\mathbf 0,t}(y)$\,. 

The next theorem proves  Theorem \ref{the:inv}.

\begin{theorem}
  \label{the:idemp_inv}
Under   the hypotheses of  Theorem \ref{the:inv},
 the sequence of the distributions of $X^{(s)}_{n}$ LD converges
at rate $b_n^2$ to deviability $\hat \Pi$ on $\R$ such that
\begin{equation}
  \label{eq:10'}
  \hat\Pi(y)=\lim_{t\to\infty}\hat\Pi_{t}(y)\,, y\in\R\,.
\end{equation}
Furthermore, 
for any bounded continuous nonnegative function $f$ on $\R$ and
compact $B\subset\R\times\bbC_0(\R_+,\R)$\,,
\begin{equation}
  \label{eq:11'''}
    \lim_{t\to\infty}\sup_{(x,\phi)\in B}\abs{
\sup_{y\in\R}f(y)\Pi_{x,\phi,t}(y)-
\sup_{y\in\R}f(y)\hat\Pi(y)}=0\,.
\end{equation}
\end{theorem}

A proof outline for Theorem \ref{the:idemp_inv} is as follows.
 It is  observed that, for any initial condition $(x,\phi)$\,,
 the idempotent process $X$ comes arbitrarily close to
$-\beta$ in the long term (Lemma \ref{le:limit_dynamics}), which leads
to a coupling of $X$ and $\hat X$\,.
It is observed also that $\hat\Pi_t(y)$ is a monotonically increasing
function of $t$\,, so, the  limit in
\eqref{eq:10'} exists.
Owing to the coupling, there exists a
unique long term idempotent distribution of $X$ (Lemma 
\ref{le:idemp_meas_limit}),  no
matter $(x,\phi)$\,.
Any  LD limit point in distribution of 
$\{ X_n^{(s)}\,,n\in\N\}$ is shown to be the
long term idempotent distribution of $X$ (Lemma \ref{le:stat_LDP}).

\begin{lemma}
  \label{le:limit_dynamics}
As $t\to\infty$\,,
$  X(t)\to-\beta \;\; \text{$\Pi$--a.e.}$ uniformly over $(x,\phi)$
 from compact subsets
of $\R\times\bbC_0(\R_+,\R)$\,.
\end{lemma}
\begin{proof}
By Lemma \ref{le:renewal_limit}, it suffices to
  prove that, for $\Pi$--almost every trjectory $(w,k)$ of $(W,K)$\,,
  \[
  \lim_{t\to\infty}\bl(\int_0^t\bl(1-F(t-s)\br)\sigma\,\dot w(s)\,ds
+\int_{\R_+^2} \ind{x'+s\le t}\,\dot k(\mu s,F(x'))\,dF(x')\,\mu \,ds\br)=0\,.
  \]
By the Cauchy--Schwarz inequality and 
 the fact that $\int_0^\infty(1-F(s))\,ds=1/\mu,$
\begin{multline*}
  \abs{\int_0^t\bl(1-F(t-s)\br)\sigma\,\dot w(s)\,ds}\le
\sqrt{\int_0^t\bl(1-F(t-s)\br)\,ds}
\sqrt{\int_0^t\bl(1-F(t-s)\br)\sigma^2\dot w(s)^2\,ds}\\
\le\mu^{-1/2}\sigma\sqrt{\int_0^t\bl(1-F(t-s)\br)\dot w(s)^2\,ds}\,.
\end{multline*}
Since $1-F(t-s)\to0$\,, as $t\to\infty$\,, and
$\int_0^\infty\dot w(s)^2\,ds<\infty$ $\Pi$-a.e., by  dominated convergence,
\[
  \lim_{t\to\infty}\int_0^t\bl(1-F(t-s)\br)\sigma\,\dot w(s)^2\,ds=0\,.
\]
If $(k(x',t))$ is a trajectory of the Kiefer idempotent  process $K$\,, then
$0=k(1,\mu t)=\int_0^t\int_0^1\dot k(x',\mu s)dx'\mu ds$\,.
Hence,
\begin{multline}
  \label{eq:44}
      \int_{\R_+^2} \ind{x'+s\le t}\,\dot k(F(x'),\mu s)
\,dF(x')\,\mu\,ds=
  \int_0^{ t}\int_0^{F(t-s)}\dot k(x',\mu s)\,dx'\,\mu\,ds
\\=-\int_0^{ t}\int_{F(t-s)}^1\dot k(x',\mu s)\,dx'\,\mu\,ds=
-\int_0^{ t}\int_0^1\ind{ x'\ge F(t-s)}\dot k(x',\mu s)\,dx'\,\mu\,ds\,.
\end{multline}
By the Cauchy--Schwarz inequality,
\begin{multline*}
    \bl(\int_{[0,1]\times\R_+}\ind{s\le t}\ind{ x'\ge F(t-s)}\dot
  k(x',\mu s)\,dx'\,\mu\,ds\br)^2 \\
\le \int_{[0,1]\times\R_+}\ind{s\le t}\ind{ x'\ge F(t-s)}\dot k(x',\mu
s)^2\,dx'\,\mu\,ds\,.
\end{multline*}

Since $\ind{s\le t}\ind{x'\ge F(t-s)}\to0$\,, as $t\to\infty$\,, unless
$x'=1$\,,
and $\int_{[0,1]\times\R_+}\dot
  k(x',s)^2\,dx'\,ds<\infty$ $\Pi$-a.e., by dominated convergence,
$\int_{[0,1]\times\R_+}\ind{s\le t}\ind{x'\ge F(t-s)}\dot
  k(x',\mu s)^2\,dx'\,\mu\,ds\to0$\,, as $t\to\infty$\,,  $\Pi$-a.e.
It follows that\,, $\Pi$-a.e.,
\[
  \lim_{t\to\infty}\int_{\R_+^2} \ind{x'+s\le t}\,\dot k(F(x'),\mu s)\,dF(x') \,\mu\,ds=0\,. 
\]
\end{proof}
  \begin{lemma}
  \label{le:idemp_meas_limit}
For $y\in\R$\,,  the  limit in \eqref{eq:10'} exists.
Furthermore, 
for any bounded continuous nonnegative function $f$ on $\R$\,, 
uniformly over  $(x,\phi)$ from compact
subsets of $\R\times\bbC_0(\R_+,\R)$\,,
\begin{equation}
  \label{eq:11''}
    \lim_{t\to\infty}
\sup_{y\in\R}f(y)\Pi_{x,\phi,t}(y)=
\sup_{y\in\R}f(y)\hat\Pi(y)\,.
\end{equation}
\end{lemma}
\begin{proof}
Note that 
 $\hat\Pi_{t}(y)$ is a nondecreasing function of $t$ because
   ''sitting at $-\beta$ incurs zero cost''. In more detail,
 given $T>0$\,,
 one can associate with $\hat X$\,,  $W$\,,
 and $K$ a delayed trajectory 
$(\hat X_T,W_T,K_T)$ such that
$\hat X_T(t)=-\beta\ind{t< T}+\ind{t\ge T}\hat X(t-T)$\,,
$\dot W_T(t)=\ind{t\ge T}\dot W(t-T)$\,, 
 and
$\dot K_T(x',\mu t)=\ind{t\ge T}\dot K(x',\mu(t-T))$\,.
  Then, \eqref{eq:137} holds
for
$\hat X_T(t)$\,,  $W_T(t)$ and $K_T(F(x'), t)$\,.\footnote{In
some  more
  detail, 
$\int_0^t(1-F(t-s))\dot W_T(s)\,ds=
\int_T^t(1-F(t-s))\dot W(s-T)\,ds=
\int_0^{t-T}(1-F(t-s-T))\dot W(s)\,ds$
and $\int_0^t\hat X_T(t-s)^+\,dF(s)=\int_0^{t-T}\hat X(t-T-s)^+\,dF(s)
$\,.
Also, 
$\int\ind{x'+s\le t}\dot{k}_T(F(x'),\mu s)\,dF(x')\,ds=
\int_0^t\int_0^{F(t-s)}\dot{k}_T(u,\mu s)\,du\,ds=
\int_T^t\int_0^{F(t-s)}\dot{k}(u,\mu(s-T))\,du\,ds=
\int_0^{t-T}\int_0^{F(t-T-s)}\dot{k}(u,\mu s)\,du\,ds=
\int\ind{x'+s\le t-T}\dot{k}(F(x'),\mu s)\,dF(x')\,ds$\,.}
 In
addition,
\[
\int_0^\infty \dot w_T(t)^2\,dt\le\int_0^\infty \dot w(t)^2\,dt
\text{ and }\int_0^\infty\int_0^1 \dot k_T(x',t)^2\,dx'\,dt\le
\int_0^\infty\int_0^1 \dot k(x',t)^2\,dx'\,dt
\]
 so that $\Pi(\hat X=q)\le\Pi(\hat X=q_T)$\,, where $q_T(t)=q(t-T)\ind{t\ge
   T}$\,.
Hence,  
$\hat \Pi_{t+T}(y)\ge
\hat  \Pi_{t}(y)$\footnote{$\hat  \Pi_{t}(y)=\Pi(\hat X(t)=y)=
\sup_{q:\,q(t)=y}\Pi(\hat X=q)\le
\sup_{q:\,q(t)=y}\Pi(\hat X=q_T)=
\sup_{q:\,q_T(T+t)=y}\Pi(\hat X=q_T)\le \sup_{q:\,q(T+t)=y}\Pi(\hat X=q)=
\Pi(\hat X(T+t)=y)=\hat\Pi_{T+t}(y)$\,.}, which implies the existence of the
 limit in \eqref{eq:10'}.

The limit \eqref{eq:11''}
 is proved via a coupling argument. 
Let 
$\kappa>0$ and $B$ represent a compact subset of $\R\times\bbC_0(\R_+,\R)$\,. For
  $(x,\phi)\in B$\,, let  
$K_{x,\phi}=\{q\in\bbC(\R_+,\R):\, \Pi_{x,\phi}(q)\ge\kappa\}$\,.
The set $K_{x,\phi}$ is compact. Furthermore, the set $\cup_{(x,\phi)\in B}K_{x,\phi}$ is
compact too.
 For $\epsilon>0$\,, $u>0$ and 
$q\in\bbC(\R_+,\R)$\,, let $T_{\epsilon,u}(q)=\inf\{s:\,
\abs{q(s')+\beta}\le\epsilon\text{ for all }s'\in[s,s+u]\,\}$\,.  By Lemma \ref{le:limit_dynamics},
$T_{\epsilon,u}(q)<\infty$  when $q\in K_{x,\phi}$\,.
Let $\tilde T_{\epsilon,u}(x,\phi)=\sup_{q\in K_{x,\phi}}T_{\epsilon,u}(q)$\,. 
By a similar argument as in the proof of part 1 of Lemma 3.1 in
Puhalskii \cite{Puh23a}, the latter supremum is attained and
 the function $\tilde T_{\epsilon,u}(x,\phi)$
is upper semicontinuous in $(x,\phi)$\,. Let $\hat T_{\epsilon,u}=\max_{(x,\phi)\in B}\tilde
T_{\epsilon,u}(x,\phi)$\,.

It is proved, first, that,
given arbitrary $\kappa>0$\,, there exists 
$\epsilon>0$ such that, for all $(x,\phi)\in B$\,, all $y$ and 
$t\ge \hat T_{\epsilon,u}+u$\,, 
\begin{equation}
  \label{eq:232}
  \Pi_{x,\phi,t}(y)\le \hat\Pi_{t}(y)+\kappa\,.
\end{equation}
Suppose  that $\Pi_{x,\phi,t}(y)\ge\kappa$
and let $q$ represent a trajectory such that 
\begin{multline}
  \label{eq:114}
            q(s)=(1-F(s))x^+ -(1-F^{(0)}(s))x^-+\phi(s)-\beta F^{(0)}(s)
+\int_0^sq(s-s')^+\,dF(s')
\\
+\sigma\int_0^s\bl(1-F(s-s')\br)\,\dot w(s')\,ds'
+\int_0^s\int_{0}^s \ind{x'+s'\le s}\,\dot k(F(x'),\mu s')\,dF(x')\,\mu \,ds'\,,
\end{multline}
for certain $w$ and $k$\,, $q(t)=y$\,, 
and $\Pi(q)=
\Pi_{x,\phi,t}(y)$\,.
Introduce  
\[
\Delta(s)= \sigma\int_0^s\bl(1-F(s-s')\br)\dot w(s')\,ds'
+\int_{\R_+^2} \ind{x'+s'\le s}\,\dot k(F(x'),\mu s')\,dF(x')\,\mu
\,ds'\]
and, for $s\ge T_{\epsilon,u}(q)$\,,
\begin{multline}
  \label{eq:91}
  \varrho         (s)
=(1-F(s))x^+-(1-{F^{(0)}}(s))(x^--\beta)+\phi(s)\\
+\int_{s- T_{\epsilon,u}(q)}^sq(s-s')^+\,dF(s')\,.
\end{multline}
By \eqref{eq:114}, on noting that $q(s')^+=0$ on $[
T_{\epsilon,u}(q), T_{\epsilon,u}(q)+u]$\,, provided
$\epsilon<\beta$\,,
so that $\int_{s-( T_{\epsilon,u}(q)+u)}^{s- T_{\epsilon,u}(q)}q(s-s')^+\,dF(s')=0$\,, for $s\ge  T_{\epsilon,u}(q)+u$\,,
\begin{equation}
  \label{eq:84}
q(s)=\varrho(s)-\beta  +\int_0^{s-(
  T_{\epsilon,u}(q)+u)}q(s-s')^+\,dF(s')
+\Delta(s)\,.
\end{equation}
Let  $\tilde s=s-( T_{\epsilon,u}(q)+u)$\,. Introduce
\begin{equation}
  \label{eq:4'}
   \tilde q(\tilde s)=q(s)\,,\;
\tilde\varrho(\tilde s)=\varrho(s) \,,\;
 \tilde \Delta(\tilde s)=\Delta(s)\,.
\end{equation}

By \eqref{eq:84}, for $ s\ge  T_{\epsilon,u}(q)+u$\,,
\begin{equation}
  \label{eq:93}
   \tilde q( \tilde s)=\tilde \varrho(\tilde s)-\beta+
\int_0^{\tilde s} \tilde q(\tilde
  s-s')^+\,dF(s')+ \tilde \Delta(\tilde s)\,.
\end{equation}
Note that, for $s\ge T_{\epsilon,u}(q)+u$\,,
\[
  \int_{s- T_{\epsilon,u}(q)}^sq(s-s')^+\,dF(s')\le
\int_{u}^s q(s-s')^+\,dF(s')\le\sup_{s'\in\R_+}q(s')^+(1-F(u))\,.
\]
Since the latter $\sup$ is  bounded over $(x,\phi)\in B$
by $q(s')$ converging to $-\beta$ uniformly,
 as $s'\to\infty$\,,  (by Lemma \ref{le:limit_dynamics}), 
uniformly over $(x,\phi)\in  B$\,,
\[
  \lim_{u\to\infty}\sup_{s\ge T_{\epsilon,u}(q)+u}
 \int_{s- T_{\epsilon,u}(q)}^sq(s-s')^+\,dF(s')=0\,.
\]
Hence, by \eqref{eq:91} and the fact that $s=\tilde
s+T_{\epsilon,u}(q)+u$\,,
uniformly over $(x,\phi)\in B$\,,
\begin{equation}
  \label{eq:66}
      \lim_{u\to\infty}\sup_{\tilde s\in\R_+}
\tilde\varrho(\tilde s)=0\,.
\end{equation}
 It is also noteworthy that, since $\tilde
 q(0)=q(T_{\epsilon,u}(q)+u)$ is within $\epsilon$ to $-\beta$\,, it
 may be assumed, owing to \eqref{eq:93}, that
$   \abs{\tilde\Delta(0)}\le 2\epsilon\,.$

Let 
\begin{equation}
  \label{eq:26}
    \tilde q_\epsilon(\tilde s)=
\pm\epsilon-\beta+\int_0^{\tilde s}\tilde q_\epsilon(\tilde
  s-s')^+\,dF(s')+ \tilde \Delta(\tilde s)\,,
\end{equation}
where a plus or a minus sign is chosen 
according as $y\gtrless -\beta$\,.
(Since $\hat\Pi_{t}(-\beta)=1$\,,  \eqref{eq:232}
needs proof only when $y\not=-\beta$\,.)
The existence and uniqueness  of an essentially locally bounded
solution $\tilde q_\epsilon(t)$ to \eqref{eq:26} follows from Lemma B.2 in Puhalskii and
Reed \cite{PuhRee10}.
Define, accordingly, $q_\epsilon(s)=\tilde q_\epsilon( s -T_{\epsilon,u}(q)-u)$\,.
By \eqref{eq:93}, \eqref{eq:66},  and Lemma
\ref{le:monot}, for $u$ great enough,
$\tilde q(\tilde s)\lesseqgtr \tilde q_\epsilon(\tilde s)$ so that,
for $s\ge T_{\epsilon,u}(q)+u$\,, 
\begin{equation}
  \label{eq:177}
   q( s)\lesseqgtr  q_\epsilon( s)\,.
\end{equation}

Now a coupling of $ q_\epsilon$  with a trajectory emanating from
$-\beta$ is produced. The latter stays put until time
$T_{\epsilon,u}(q)+u-1$ and
gets 
 to $\pm\epsilon-\beta+\tilde\Delta(0) $ at $ T_{\epsilon,u}(q)+u $\,.
 More specifically, let
 \begin{multline}
   \label{eq:51}
\hat q_{\epsilon}(s)=
    -\beta\ind{s\in[0, T_{\epsilon,u}(q)+u-1]}\\+
(-\beta+\tilde\Delta(0)\pm\epsilon)
\displaystyle
\frac{ \int_{ T_{\epsilon,u}(q)+u-1}^{s}
(1-F(s-s'))ds'}
{\int_{0}^{1}
(1-F(s'))\,ds'}\ind{s\in[ T_{\epsilon,u}(q)+u-1, T_{\epsilon,u}(q)+u]}\\+
 q_\epsilon(s)\ind{s\ge T_{\epsilon,u}(q)+u}\,,
 \end{multline}
\begin{subequations}
  \begin{align}
   \label{eq:155}
  \dot{\hat w}(s)=&(\tilde\Delta(0)\pm\epsilon)\frac{\ind{s\in
[ T_{\epsilon,u}(q)+u-1, T_{\epsilon,u}(q)+u]}}{\sigma\int_{0}^{1}
(1-F(s'))\,ds'}+
\dot w(s)\ind{s\ge T_{\epsilon,u}(q)+u}\,,\\\notag\text{ and } \\  \label{eq:155''}
  \dot{\hat k}(F(x'),\mu s)=&\,
\dot k(F(x'),\mu s)\ind{s\ge T_{\epsilon,u}(q)+u}\,.
\end{align}
\end{subequations}
Note that $\hat q_{\epsilon}(s)\le0 $\,, for $s\le
T_{\epsilon,u}(q)+u$\,, provided $\epsilon$ is small enough.
Owing to \eqref{eq:26}, \eqref{eq:51}, \eqref{eq:155} and \eqref{eq:155''},
\begin{multline*}
  \hat  q_\epsilon(s)=-\beta+\int_0^s\hat q_{\epsilon}(s-s')^+\,dF(s')
+\int_0^s\bl(1-F(s-s')\br)\sigma\,\dot {\hat w}(s')\,ds'
\\+\int_{\R_+^2} \ind{x'+s'\le s}\,\dot{\hat k}(F(x'),\mu s')\,dF(x')\,\mu \,ds'
\,.
\end{multline*}

By the definition of $I^Q(q)$ in \eqref{eq:2} and \eqref{eq:1}, by 
 \eqref{eq:155} and \eqref{eq:155''}, for $\epsilon$ small enough,
\begin{multline*}
  I^Q(\hat q_{\epsilon})\le \frac{1}{2}\bl(
\int_0^\infty \dot{\hat w}(s)^2\,ds
+\int_0^\infty\int_0^1 \dot{\hat  k}(x',s)^2\,dx'\,ds\br)
\\\le \frac{1}{2}\bl(
\frac{(\tilde\Delta(0)\pm\epsilon)^2}{\int_0^1(1-F(s'))\,ds'}+\int_0^\infty \dot{ w}(s)^2\,ds
+\int_0^\infty\int_0^1 \dot{  k}(x',s)^2\,dx'\,ds\br)\le
\kappa+I^Q( q)\,.
\end{multline*}
Hence, $\Pi(\hat q_{\epsilon})\ge e^{-I^Q(\hat q_{\epsilon})}\ge  e^{-I^Q(q)}e^{-{\kappa}}
=\Pi(q)e^{-{\kappa}}\ge \Pi(q)-\kappa$ so that
$\Pi_{x,\phi,t}(y)=\Pi(q)\le \Pi(\hat q_{\epsilon})+\kappa
$\,. 

On recalling that $q(t)=y$\,, 
the inequality in \eqref{eq:177} yields the inequality
\[
  y\lesseqgtr  
\hat q_\epsilon(t)\,.
\]
Hence, $\hat q_\epsilon(\hat t_\epsilon)=y$\,, for some $\hat
t_\epsilon\le t$\,.
By monotonicity, 
$\hat \Pi_t(y)\ge \hat\Pi_{\hat t_\epsilon}(y)\ge\Pi(\hat
q_\epsilon)\ge\Pi_{x,\phi,t}(y)-\kappa$\,, implying \eqref{eq:232}.
It follows that, for all  sets $F\subset \bbC(\R_+,\R)$\,,
\begin{equation}
  \label{eq:61}
  \limsup_{t\to\infty}\sup_{(x,\phi)\in B}\Pi_{x,\phi,t}(F)\le\hat\Pi(F)\,.
\end{equation}

It is now proved that,
given arbitrary $\kappa>0$\,, $t>0$\,,  $\delta>0$\, and $y\in\R$\,,
  uniformly over $(x,\phi)\in B$\,,
\begin{equation}
  \label{eq:232'}
 \hat\Pi_{t}(y)\le
\liminf_{t'\to\infty}
 \Pi_{x,\phi,t'}(B_\delta(y))+\kappa\,,
\end{equation}
where $B_\delta(y)$ represents the open $\delta$--ball about $y$\,.
Let $\hat q$ represent a trajectory such that $\hat q(0)=-\beta$\,, $\hat
q(t)=y$ and $\Pi(\hat q)=\hat\Pi_t(y)\ge\kappa$\,. It solves
\begin{equation}
  \label{eq:64}
  \hat q(s)= -\beta
+\int_0^s\hat q(s-s')^+\,dF(s')
+\hat\Delta(s)\,,
\end{equation}
where
\begin{multline*}
  \hat\Delta(s)=
\sigma\int_{0}^s\bl(1-F(s-s')\br)\,\dot{\hat w}(s')\,ds'\\
+\int_{0}^s\int_{0}^s \ind{x'+s'\le s}\,\dot{\hat k}(F(x'),\mu
s')\,dF(x')\,
\mu \,ds'\,.
\end{multline*}
Let $ \overline{q}$ solve \eqref{eq:114} with $\dot w(s')=\dot
k(x',s')=0$\,.
Let $\overline{q}'(s)=\overline{q}(s)$\,, for $s\le T_{\epsilon,u}(\overline{q})+u$\,.
Let, for $s\ge T_{\epsilon,u}(\overline{q})+u$\,,
\begin{multline}
  \label{eq:81}
   \overline{q}'(s)=(1-F(s))x^+ -(1-F^{(0)}(s))x^-+\phi(s)-\beta F^{(0)}(s)\\
+\int_0^s\overline{q}'(s-s')^+\,dF(s')+
\hat\Delta_{T_{\epsilon,u}(\overline{q})+u}( s)\,.
\end{multline}
Consequently,
$\Pi(\overline{q}')=\Pi(\hat q)=\hat \Pi_t(y)\ge\kappa$\,.

Let
$\tilde{\overline s}=s-(T_{\epsilon,u}(\overline{q})+u)$\,,
 $\tilde{\overline q}(\tilde{\overline s} )=\overline{q}'(s)$ and
 $\tilde{\hat \Delta}(\tilde{\overline s} )
=\hat\Delta_{T_{\epsilon,u}(\overline{q})+u}( s)$\,.
Then, in analogy with \eqref{eq:93}, provided $\tilde{\overline s} \ge0$\,,
\[
  \tilde{\overline q}(\tilde{\overline s})=\tilde{\overline \varrho}
(\tilde{\overline  s})-\beta
+\int_{0}^{\tilde{\overline s}}\tilde{\overline q}(\tilde{\overline s}-s')^+\,dF(s')
+\hat\Delta(\tilde{\overline s})\,,
\]
where
$\tilde{\overline \varrho}(\tilde{\overline s})=\overline{\varrho}
(s)$ and
 (cf. \eqref{eq:91})
\[
  \overline \varrho         (s)
=(1-F(s))x^+-(1-{F^{(0)}}(s))(x^--\beta)+\phi(s)
+\int_{s- T_{\epsilon,u}(\overline q)}^s\overline q(s-s')^+\,dF(s')\,.
\]
 In analogy with
\eqref{eq:66}  $\sup_{\tilde{\overline s}\in\R_+}\tilde{\overline \varrho}(\tilde{\overline s})\to0$\,, as
$u\to\infty$\,, uniformly over $(x,\phi)\in B$\,.
Hence, for arbitrary $\epsilon>0$ and $u$ great enough,
$\abs{\tilde{\overline\varrho}(\tilde{\overline s})} \le \epsilon$\,. 
On recalling \eqref{eq:64} and \eqref{eq:81},
\[
  \sup_{ s'\le t}\abs{\hat q( s')-\tilde{\overline q}( s')}\le\epsilon
+\sup_{ s'\le  t}\abs{\hat q(s')-\tilde{\overline q}(s')}F(t)\,.
\]
It follows that
\[
  \abs{\hat q(t)-\tilde{\overline q}(t)}\le\frac{\epsilon}{1-F(t)}\,.
\]
Assuming that $\epsilon/(1-F(t))\le\delta$ implies that
\[
  \abs{\hat q(t)-\overline q'(t+(T_{\epsilon,u}(\overline{q})+u))}\le\delta\,.
\]
Thus, for all $u$ great enough and all $(x,\phi)\in B$\,,
 \[\Pi_{x,\phi,T_{\epsilon,u}(\overline{q})+u+ t}(B_\delta(y))\ge
\hat\Pi_t(y)\,.\]
The quantity $T_{\epsilon,u}(\overline{q})$ is seen to be a continuous
function of $u$\,.\footnote{Let $u_n\to u$ and 
$T_{\epsilon,u_n}(\overline{q})\to T$\,. Then $\abs{\overline{q}(s)+\beta}\le \epsilon$ on
$[T,T+u]$ so that $T\ge T_{\epsilon,u}(\overline{q})$\,.
If $T> T_{\epsilon,u}(\overline{q})$\,, then
$T_{\epsilon,u}(\overline{q})<T_{\epsilon,u_n}(\overline{q})-\delta$\,,
for some $\delta>0$ and all $n$ great enough, which contradicts the definition of
$T_{\epsilon,u_n}(\overline{q})$\,. }
The intermediate value theorem implies that, for
any $t'$ great enough, there exists $u'$ such that 
$T_{\epsilon,u'}(\overline{q})+u'+ t=t'$\,,
 so, \eqref{eq:232'} follows.
Letting $t\to\infty$ in \eqref{eq:232'} obtains that,
 for all open  sets $G$\,,
 uniformly over
$(x,\phi)\in B$\,,
\begin{equation}
  \label{eq:60}
  \liminf_{t\to\infty}\Pi_{x,\phi,t}(G)\ge \hat\Pi(G)\,.
\end{equation}

The inequalities \eqref{eq:61} and \eqref{eq:60} imply the assertion
of the lemma.
  \end{proof}
The next lemma concludes the proof of Theorem   \ref{the:idemp_inv}.
\begin{lemma}
  \label{le:stat_LDP}
Under the hypotheses of  Theorem \ref{the:inv},
the sequence of the
  distributions of  $X_n^{(s)}$ LD converges at rate $b_n^2$ to
 $\hat\Pi$\,, the latter set function being a deviability.
\end{lemma}
\begin{proof}
Let $\breve U_n(t)=(\tilde U_n(x,t)\,, x\in\R_+)
$\,, as defined in \eqref{eq:28}.
Let  $\Psi_n(t)=(X_n(t),\breve U_n(t))\,$ and
$\Psi_n=(\Psi_n(t)\,,t\in\R_+)$\,. It is a
stationary process with values in
 $S=\R\times\D(\R_+,\R)$\,. Generic elements of $S$ are denoted by
 $\psi=(x,\phi)$\,, where $x\in\R$ and $\phi\in\D(\R_+,\R)$\,.
By  stationarity, for  $\Gamma\subset\R$\,,
\begin{equation}
  P(X_n(0)\in\Gamma)
=\int_{S} P(X_n(t)\in
\Gamma|\Psi_n(0)=\psi)\,P(\Psi_n(0)
\in d\psi)
\,.
  \label{eq:235}
\end{equation}
  LD limits are  taken next on both sides of the latter equation.
By respective Theorems \ref{the:stat_exp_tight} and \ref{the:emp_cdf}, 
 the sequences of
 the distributions of $X_n(0)$ 
and $\breve U_n(0)$ are exponentially tight and
 $\bbC_0$--exponentially tight, respectively.
Therefore, the sequence of the distributions of $\Psi_n(0)$ is
exponentially tight too. Let $\tilde\Pi$ represent a deviability
on $\R\times \bbC_0(\R_+,\R)$ that
is an LD limit point of the distributions of 
$\Psi_n(0)$\,. Thus, with $(X(0),\Phi)$ being the canonical idempotent
variable on $\R\times \bbC_0(\R_+,\R)$ endowed with $\tilde\Pi$\,,
it can  be assumed that, for some
subsequence 
$n'$\,, $(X_{n'}(0),\tilde U_{n'}(0))$ LD converges to
$(X(0),\Phi)$\,.

The hypotheses of Lemma \ref{le:LDconv_mixt} are checked next for the
mixture on the righthand side of \eqref{eq:235}.
 Suppose that $\psi_{n'}=(x_{n'},\phi_{n'})\to \psi=(x,\phi)$\,, where 
$\psi_{n'}$
belongs to the support of the distribution of
 $\Psi_{n'}(0)$ and $\tilde\Pi(\psi)>0$\,.
More explicitly,
$\psi_{n'} $ assumes values in 
the set 
\begin{align*}
  \Delta_{n'}=\{\sqrt{n'}/b_{n'}
(l/n'-1),(\sqrt{n'}/b_{n'}(1-\hat F(x)
-(1-F^{(0)}(x)))\,,x\in\R_+):\,l\in \{0\}\cup\N\,,\\ \text{
 $\hat F=(\hat F(x)\,,x\in\R_+)$ representing a  discrete cdf on 
$\R_+$ with jumps being multiples of $1/n'$}\}\,.
\end{align*}
In addition, $\phi\in\bbC_0(\R_+,\R)$\,.
As follows from \eqref{eq:67}, the   distribution at $t$ of the
 solution to the
equation 
\begin{multline*}
    X_{n'}(s)=(1-F(s))x_{n'}^+-(1-F^{(0)}(s)) x_{n'}^-+
\phi_{n'}(s)+
\int_0^sX_{n'}(s-s')^+\,dF(s')
\\+H_{n'}(s)-\Theta_{n'}(s)
\end{multline*}
 is  a regular version of 
the conditional distribution
of $X_{n'}(t)$ given $\Psi_{n'}(0)=\psi_{n'}$\,.
 Denote it by $P_{n',\psi_{n'},t}$\,.
By part 2 of Theorem \ref{the:heavy_traffic} (with $\phi_n$ as $\Phi_n$ and
$\phi$ as $\Phi$)\,,  $X_{n'}$
 LD converges in distribution at rate
$b_{n'}^2$ to the idempotent process $X=(X(s)\,,s\in\R_+)$ such that 
\begin{multline*}
           X(s)=(1-F(s))x^+- (1-F^{(0)}(s))x^{-}
+\phi(s)-\beta F^{(0)}(s)
+\int_0^sX(s-s')^+\,dF(s')
\\+\sigma\int_0^s\bl(1-F(s-s')\br)\,\dot W(s')\,ds'
+\int_0^s\int_{0}^s \ind{x'+s'\le s}\,\dot K(  F(x'),\mu s')\,dF(x')\,\mu\,ds'
\,.
\end{multline*}
 Hence, $P_{n',\psi_{n'},t}$ LD converges to
the idempotent distribution of $X$ at $t$\,, which is 
 $\Pi_{x,\phi,t}$\,. Since the mapping $(W,K)\to X$ is continuous
 $\Pi$--a.e.,
\[
   \Pi_{x,\phi,t}(y)=\sup_{(w,k)\in \Sigma_{x,\phi,y,t}}\Pi^W(w)\Pi^K(k)\,,
 \]
where $\Sigma_{x,\phi,y,t}$ represents the set of
 $(w,k)\in\Upsilon$ such that $q(t)=y$\,, with $(q(s)\,,s\in[0,t])$ being
defined by the equation
\begin{multline*}
             q(s)=(1-F(s))x^+- (1-F^{(0)}(s))x^{-}
+\phi(s)-\beta F^{(0)}(s)
+\int_0^sq(s-s')^+\,dF(s')
\\+\sigma\int_0^s\bl(1-F(s-s')\br)\,\dot w(s')\,ds'
+\int_0^s\int_{0}^s \ind{x'+s'\le s}\,\dot k(F(x'),  \mu s')\,dF(x')\,\mu\,ds'
\,.
\end{multline*}
The deviability density $\Pi_{x,\phi,t}(y)$ is seen to be
  upper semicontinuous in
$(x,\phi,y)$\,. Indeed, let $(x_{\tilde n},\phi_{\tilde n},y_{\tilde n})\to (x,\phi,y)$\,.
The sets $\Sigma_{x_{\tilde n},\phi_{\tilde n},y,t_{\tilde n}}$ being closed implies the
existence of $(w_{\tilde n},k_{\tilde n})\in \Sigma_{x_{\tilde n},\phi_{\tilde n},y,t_{\tilde n}}$ such that 
$\Pi_{x_{\tilde n},\phi_{\tilde n},t}(y_{\tilde n})=\Pi^W(w_{\tilde n})\Pi^K(k_{\tilde n})$\,. Provided
$\limsup_{n\to\infty}\Pi^W(w_{\tilde n})\Pi^K(k_{\tilde n})>0$\,, it may be assumed 
by passing to subsequences if necessary that 
$w_{\tilde n}\to w$ and $k_{\tilde n}\to k$\,, where  $(w,k)\in
\Sigma_{x,\phi,y,t}$\,. Since
\[\limsup_{\tilde n\to\infty}\Pi^W(w_{\tilde n})\le \Pi^W(w)\] and 
\[\limsup_{\tilde n\to\infty}\Pi^K(k_{\tilde n})\le \Pi^K(k)\,,\]
it follows that
\[\limsup_{\tilde n\to\infty}
\Pi_{x_{\tilde n},\phi_{\tilde n},t}(y_{\tilde n})=\limsup_{\tilde n\to\infty}\Pi^W(w_{\tilde n})\Pi^K(k_{\tilde n})\le \Pi(w)\Pi(k)\le 
\Pi_{x,\phi,t}(y)\,.\]
A similar line of reasoning shows that
   the deviability
 $(\Pi_{x,\phi,t}(\Gamma)\,,\Gamma\subset \R)$ is tight uniformly
over $(x,\phi)$ from compact subsets of $\R\times\bbC_0(\R_+,\R)$\,, i.e.,
the set $\cup_{(x,\phi)}\{y:\,\Pi_{x,\phi,t}(y)\ge a\}$ is compact for
all $a>0$\,, where the union is over $(x,\phi)$ from a compact.

By Lemma \ref{le:LDconv_mixt},
the mixture of probabilities
 on the righthand  side of \eqref{eq:235} LD converges
along $n'$ to 
the deviability
$(\sup_{\psi\in S}\Pi_{\psi,t}(\Gamma)\tilde\Pi(\psi)\,, \Gamma\subset\R)$\,.
Therefore,
for bounded nonnegative continuous function $f$\,,
\begin{equation}
  \label{eq:191}
    \sup_{y\in\R}f(y)\tilde\Pi(X(0)=y)=
\sup_{y\in\R}f(y)\sup_{\psi\in S}
\Pi_{\psi,t}(y)\tilde\Pi(\psi)\,.
\end{equation}
By Lemma \ref{le:idemp_meas_limit} and the set $\{\psi:\,
\tilde\Pi(\psi)\ge a\}$ being compact, where $a>0$\,,
the righthand side of \eqref{eq:191} converges, as $t\to\infty$\,,
to
$\sup_{y\in\R}f(y)
\hat \Pi(y)
$
so that
\[\sup_{y\in\R}f(y)\tilde\Pi(X(0)=y)=
\sup_{y\in\R}f(y)
\hat \Pi(y)\,,\] which implies that
$\tilde\Pi(X(0)=y)=\hat\Pi(y)$\,.
 Thus, $\hat\Pi$ is a
deviability, \eqref{eq:11'''} holds
  and the sequence of the distributions of $X^{(s)}_n$
LD converges to  $\hat\Pi$\,.
\end{proof}

\appendix

\section{ Moment bounds for counting renewal processes}
\label{sec:higher}
Let $(A(t)\,, t\in\R_+)$ represent an ordinary  counting renewal process
of rate $\lambda$\,, let $\xi$ represent a generic inter--arrival time
and let $(A'(t)\,, t\in\R_+)$ represent the associated equilibrium counting
renewal process.
Let  $(EA(t))^{\ast(m)}$ represent the $m$th convolution power of $EA(t)$\,.
\begin{lemma}
  \label{le:conv_moments}
 For all $m\in\N$ and $t\in\R_+$\,,
\[
 ( EA(t))^{\ast(m)}\le \frac{(1+ \lambda t)^m}{m!}
(1+
 m E(\lambda\xi)^2)^m\,.
\]
\end{lemma}
\begin{proof}
Let $\xi_1,\xi_2,\ldots$ represent the inter--arrival times for
$A(t)$\,.
Let 
\begin{equation}
  \label{eq:163}
  \hat \xi_i=\lambda\xi_i
\end{equation}
 and
\[
  \hat A(t)=\max\{i\in\N:\,\sum_{j=1}^i\hat\xi_j\le t\}.
\]
Then, $E\hat\xi_1=1$ and
\begin{equation}
  \label{eq:37}
  A(t)=\hat A(\lambda t)\,.
\end{equation}
It is proved by induction that 
\begin{equation}
  \label{eq:95}
   ( E\hat A(t))^{\ast(m)}\le \frac{(1+ t)^m}{m!}
\prod_{i=1}^m(1+
 i E\hat\xi_1^2)\,.
\end{equation}
By Lorden \cite{Lor70}, see also Beichelt and Fatti \cite{BeiFat02},
\begin{equation}
  \label{eq:216}
      E\hat A(t)\le t+E\hat\xi_1^2\,,
\end{equation}
which checks the required for $m=1$\,.
Suppose that
\eqref{eq:95} holds for some $m$\,.
Then, 
\begin{multline*}
  ( E\hat A(t))^{\ast(m+1)}=\int_0^t E\hat A(t-s)d
( E\hat A(s))^{\ast(m)}\le\int_0^t( (t-s)+ E\hat\xi_1^2)
d( E\hat A(s))^{\ast(m)}\\
=\int_0^t( E\hat A(s))^{\ast(m)} \,ds+ E\hat\xi_1^2
 E\hat A(t)^{\ast(m)}
\le\frac{ (1+ t)^{m+1}}{(m+1)! }\prod_{i=1}^m(1+
 i  E\hat\xi_1^2)\\
+ E\hat\xi_1^2\frac{(1+ t)^{m}}{m! }\prod_{i=1}^m(1+
 i  E\hat\xi_1^2)
\le\frac{( 1+ t)^{m+1}}{(m+1)! }\prod_{i=1}^{m+1}(1+
 i  E\hat\xi_1^2)\,.
\end{multline*}
By \eqref{eq:163}, \eqref{eq:37} and \eqref{eq:95},
\[
  E A(t)^{\ast(m)}= E\hat A(\lambda t)^{\ast(m)}
\le \frac{(1+\lambda t)^m}{m!}
\prod_{i=1}^m(1+
 i E(\lambda\xi)^2)\,.
\]
 \end{proof}
 \begin{lemma}
  \label{le:renewal_moments}
For
all ordinary counting renewal processes $(A(s)\,,s\in\R_+)$\,, all $m\in\N $ 
 and all $t\in\R_+$\,,
\begin{equation}
  \label{eq:120}
        E   A(t)^m\le 
2^{m-1} 
((1+\lambda t)^m(1+m E(\lambda\xi)^2)^m +m^m)\,.
\end{equation}
There exists $C'>1$ such that, for all inter--arrival times $\xi$\,,    all
$t\in\R_+$\,, and all $m\in\N    $\,,
\begin{equation}
  \label{eq:120a}
        E   A'(t)^m\le  C'^{m}((1+\lambda t)^m(1+m E(\lambda \xi)^2)^m
        +m^m)\,.
\end{equation}

 \end{lemma}
 \begin{proof}
      As shown in  Barlow and Proschan \cite[Lemma 4.1]{BarPro63} or 
 in Franken \cite[Lemma 3]{Fra63}, see also p.107 in
Beichelt and Fatti \cite{BeiFat02},
 for $m\in\N$\,,
\[
  E\Bl(
  \begin{array}[c]{c}
    A(t)\\m
  \end{array}\Br)=(EA(t))^{\ast(m)}\,,
\]
where the lefthand side represents the $m$th factorial moment of
$A(t)$\,, i.e.,
\[
    E\Bl(
  \begin{array}[c]{c}
    A(t)\\m
  \end{array}\Br)=\frac{1}{m!}\,E
\bl(A(t)(A(t)-1)\ldots( A(t)-m+1)\br)\,.
\]
Hence,
\[
    E   ((A(t)-m)^+)^m\le m!(EA(t))^{\ast(m)}\,.
\]
Therefore, 
\[
    E   A(t)^m\le 2^{m-1} 
( E (  (A(t)-m)^+)^m+m^m)
\le 2^{m-1} 
(m!\,(EA(t))^{\ast(m)} +m^m)\,,
\]
implying \eqref{eq:120} thanks to Lemma \ref{le:conv_moments}.

  Let $G(t)$ represent the delay distribution for $A'$ and let $\tau$
represent the delay.
Then,
\[
    EA'(t)^{m}=\int_0^t E(A'(t)^{m}|\tau=s)\,dG(s)
=\int_0^t E(1+A(t-s))^{m}\,dG(s)\le
 E(1+A(t))^{m}
\]
so that \eqref{eq:120a} follows from \eqref{eq:120}.

 \end{proof}

\begin{lemma}
  \label{le:stat_moments}
There exists $\tilde C>1$ such that, for
all ordinary counting renewal processes $(A(s)\,,s\in\R_+)$ of rate $\lambda$\,, 
  all $t\in\R_+$\,,
 and all  $p\ge2$\,,
 \begin{equation}
   \label{eq:91a}
          E  \abs{ A(t)-\lambda t}^p\le
\tilde C^p\bl((1+\lambda t)
    E(\lambda\xi)^{p+1}
+(1+\lambda t)^{p/2+1} p^{p/2+1} (E(\lambda\xi)^2)^{3p/2}
\br)\,.
 \end{equation}
A similar inequality holds for $A'$\,.
In addition,  for all $t\in\R_+$\,,
\begin{equation}
  \label{eq:214}
  E  \abs{ A'(t)-\lambda t}^2\le 2\lambda t
(\lambda^2E\xi^2+\frac{1}{2})\,.
\end{equation}
\end{lemma}
\begin{proof}
For ease of notation, suppose that $E\xi=1$ so that $\lambda=1$\,.
  Since
\[
   A(t)-  t=-1+ \sum_{i=1}^{ A(t)+1}\xi_i-t+
\sum_{i=1}^{ A(t)+1}(1- \xi_i)\,,
\]
by Jensen's inequality,
\begin{equation}
  \label{eq:118}
      E\abs{ A(t)- 
  t}^{p}\le 3^{p-1}\bl(1+ E\abs{
\sum_{i=1}^{ A(t)+1}\xi_{i}-t}^{p}+E\abs{
\sum_{i=1}^{ A(t)+1}(\xi_{i}-1)}^{p}
\br)\,.
\end{equation}
By Theorem 3 in Lorden \cite{Lor70},
\[
  E\abs{
\sum_{i=1}^{ A(t)+1}\xi_{i}-t}^{p}\le \frac{p+2}{p+1}\, 
E\xi^{p+1}\,.
\]
In order to bound the last term on the righthand side of
\eqref{eq:118},
 Theorem 1 in Borovkov and Utev
\cite{BorUte93} is used: more specifically, by p.265 there
with $c=p(p-1)2^{p-2}$\,,
\footnotemark\; 
if
 $p\ge2$\,, then
\begin{multline}
  \label{eq:217}
    E\abs{
\sum_{i=1}^{ A(t)+1}(\xi_{i}-1)}^{p}\le
2 p(p-1)2^{p-2}E\abs{\xi-1}^pE( A(t)+1)\\+
\frac{4}{p}\,2(p(p-1)2^{p-2})^2(\frac{p}{p-1})^{(p-1)(p-2)/p}(E(\xi-1)^2)^{p/2}E( A(t)+1)^{p/2}\,.
\end{multline}
\footnotetext{Proof of (11) in Borovkov 
 and Utev \cite{BorUte93}. Since
$d\abs{x}^p/dx =d(x^2)^{p/2}/dx=p/2(x^2)^{p/2-1}2x =p\abs{x}^{p-2}x$ and $d^2\abs{x}^p/dx^2
 =p(p-2)\abs{x}^{p-4}x^2+p\abs{x}^{p-2}$\,,
  for $p\ge2$\,, 
 \begin{multline*}
   \abs{x+y}^p-\abs{x}^p-p\abs{x}^{p-2}xy=\int_0^1(1-t)y^2
(p(p-2)\abs{x+ty}^{p-4}(x+ty)^2+p\abs{x+ty}^{p-2})\,dt\\\le
y^2(p(p-2)+p)(2(\abs{x}\vee\abs{y}))^{p-2}\le
p(p-1)2^{p-2}(y^2\abs{x}^{p-2}+\abs{y}^{p})\,.
 \end{multline*}
}
It follows that
\begin{multline*}
      E\abs{ A(t)- 
  t}^{p}\le
3^{p-1}\bl(1+\frac{p+2}{p+1}\, 
E\xi^{p+1}
+2 p(p-1)2^{p-2}(E\xi^p+1)(E A(t)+1)\\+
p(p-1)^22^{2p}(\frac{p}{p-1})^{(p-1)(p-2)/p}
(E(\xi-1)^2)^{p/2}E( A(t)+1)^{p/2}\br)  \,.
\end{multline*}
Thus, for some $\tilde C'>0$\,,
\[
  E\abs{ A(t)- 
  t}^{p}\le
\tilde C'^p\bl(1+ 
E\xi^{p+1}
+E\xi^pE A( t)+
(E\xi^2)^{p/2}E A( t)^{p/2}\br)  \,.
\]

 Now, \eqref{eq:120} implies that
\[        E   A(t)^{p/2}\le  2^{p/2}((1+ t)^{p/2+1}(1+p
 E \xi^2)^{p/2+1}
        +p^{p/2+1})\,.
\] The inequality   \eqref{eq:91a} 
 follows if one recalls  \eqref{eq:216}.
The case of $A'$ is tackled as in the proof of Lemma \ref{le:renewal_moments}.

For \eqref{eq:214}, note that
by (7) on p.57 in Cox \cite{cox1970renewal}, see also Whitt \cite{Whi85},
\[
  E(A'(t)-\lambda t)^2=2\lambda\,\int_0^t(EA(u)-\lambda u+\frac{1}{2})\,du\,.
\]
On recalling \eqref{eq:216},
 $EA(u)-\lambda u\le \lambda^2E\xi^2$\,.
\end{proof}

 \section{A nonlinear renewal equation}\label{renewal}
This section is concerned with the  properties of the equation
\begin{equation}
  \label{eq:100}
  g(t)=f(t)+\int_0^tg(t-s)^+\,dF(s)\,, t\in\R_+\,.
\end{equation}
It is assumed that $f(t)$ is a locally bounded measurable
function,
that $F(t)$ is a  non--lattice distribution function on $\R_+$  and that
\begin{equation}
  \label{eq:203}
   F(0)<1\,.
\end{equation}
The existence and uniqueness  of an essentially locally bounded
solution $g(t)$ to \eqref{eq:100} follows from Lemma B.2 in Puhalskii and
Reed \cite{PuhRee10}.
\begin{lemma}
  \label{le:renewal_limit}
Suppose that $g(t)$ satisfies \eqref{eq:100}.
If $f(t)\to-\alpha<0$\,,  then $g(t)\to-\alpha$\,, as $t\to\infty$\,.
 \end{lemma}
\begin{proof}
  Since $g(t-s)^+\ge0$\,, by \eqref{eq:100},
  \[
    \liminf_{t\to\infty}g(t)\ge -\alpha\,.
  \]
It needs to be proved that
\begin{equation}
  \label{eq:173}
  \limsup_{t\to\infty}g(t)\le-\alpha\,.
\end{equation}
By \eqref{eq:100},
\begin{equation}
  \label{eq:175}
  g(t)^+\le f(t)^++\int_0^tg(t-s)^+\,dF(s)\,.
\end{equation}
Let $y(t),\,t\in\R_+\,,$ represent the unique locally bounded measurable
function that solves the renewal equation
\begin{equation}
  \label{eq:176}
  y(t)=f(t)^++\int_0^ty(t-s)\,dF(s)\,.
\end{equation}
(For the existence and uniqueness of $y(t)$\,, see, e.g.,
Theorem 2.4 on p.146 in Asmussen \cite{Asm03} or Theorem
3.5.1 on p.202 in Resnick \cite{Res02}.)
With $u(t)=g(t)^+-y(t)$\,,
\begin{equation}
  \label{eq:181}
  u(t)\le \int_0^t u(t-s)\,dF(s)\,.
\end{equation}
Hence,
\begin{equation}
  \label{eq:184}
  u(t)\le \sup_{s\le t}u(s)\,F(t)\,.
\end{equation}
Similarly, for $s\le t$\,,
\begin{equation}
  \label{eq:185}
  u(s)\le \sup_{s'\le s}u(s')F(s)\le \sup_{s'\le t}u(s')F(t)\,,
\end{equation}
the latter inequality holding because $\sup_{s'\le t}u(s')\ge0$\,, as
$u(0)=0$\,,
and because $F(s)$ is nondecreasing. It follows  that
\begin{equation}
  \label{eq:192}
  \sup_{s\le t}u(s)\le F(t)\sup_{s\le t}u(s)
\end{equation}
implying that
\begin{equation}
  \label{eq:193}
  \sup_{s\le t}u(s)\le0\,,
\end{equation}
provided $F(t)<1$\,. With $t^\ast=\inf\{s:\,F(s)=1\}$\,,  for any
sequence
$t_n\uparrow t^\ast$\,, $\sup_{s\le t_n}u(s)\le0$ and $\sup_{s\le
  t_n}u(s)\to
\sup_{s< t}u(s)$ so that
$\sup_{s< t^\ast}u(s)\le0$\,. Hence,
\[
\int_0^{t^\ast} u(t^\ast-s)\,dF(s)=\int_0^{t^\ast}
u(t^\ast-s)\ind{s>0}\,dF(s)+
u(t^\ast) F(0)
\le u(t^\ast) F(0)\,.
\]
By \eqref{eq:181} and \eqref{eq:203},
$u(t^\ast)\le0$ so that \eqref{eq:193} holds  for $t=t^\ast$ too.
If $t>t^\ast$\,, then
\[
    \int_0^tu(s)\,dF(s)=\int_0^{t^\ast}u(s)dF(s)\le0\,,
\]
implying \eqref{eq:193} once again.
(An alternative way to get rid of the restriction that $F(t)<1$ in
\eqref{eq:193}
 is to
look at iterations of \eqref{eq:181}, similarly to the 
proof of   part (2)
of Theorem 3.5.1 on p.202 in Resnick \cite{Res02}.)   

It has thus been proved that
\begin{equation}
  \label{eq:204}
  \sup_{s\le t}g(s)^+\le\sup_{s\le t}y(s)\,.
\end{equation}
Since $f(t)\to-\alpha$\,,  as $t\to\infty$\,, the function $f(t)^+$ is of
bounded support, so, it is directly Riemann integrable, cf. Remark
3.10.1 on p.234  in Resnick \cite{Res02}.
By the key renewal theorem, see, e.g., Theorem 3.10.1 on p.237
along with Theorem 3.5.1 on p.302 in
Resnick \cite{Res02}, with the account of part  (1) 
of Theorem 3.5.1 on p.202 in Resnick \cite{Res02},
\[
\lim_{t\to\infty}  y(t)= \frac{1}{\int_0^\infty s\,dF(s)}\, \int_0^\infty f(s)^+\,ds<\infty\,.
\]
It follows that $\sup_{t\in\R_+}y(t)<\infty$ so that, by \eqref{eq:204},
$  \sup_{t\in\R_+}g(t)^+<\infty\,.
$
The latter inequality enables one to apply Fatou's lemma to
\eqref{eq:100} in order to obtain that
\begin{multline*}
   \limsup_{t\to\infty} g(t)\le
\limsup_{t\to\infty} f(t)+\limsup_{t\to\infty}\int_0^\infty 
g(t-s)^+\ind{s\le t}\,dF(s)\\\le -\alpha+\int_0^\infty \limsup_{t\to\infty}
g(t-s)^+\ind{s\le t}\,dF(s)\le-\alpha+ \limsup_{t\to\infty}
g(t)^+ \,,
\end{multline*}
implying \eqref{eq:173}.

\end{proof}\begin{lemma}
    \label{le:monot}
Suppose that $g_1(t)$ and $g_2(t)$ are solutions to \eqref{eq:100}
that correspond to respective functions $f_1(t)$ and $f_2(t)$ as the function
$f(t)$\,. If $f_1(t)\le f_2(t)$\,, for all $t$\,, then
$g_1(t)\le g_2(t)$\,, for all $t$\,.
  \end{lemma}
  \begin{proof}
    Use Picard iterations.
  \end{proof}
\section{$\bbC_0$--exponential tightness}\label{c_0}
 Let  
$\bbC_0(\R_+,\R)$ represent the subset of $\bbC(\R_+,\R)$ of
functions $\phi=(\phi(t)\,,t\in\R_+)$ such that $\phi(t)\to0$\,, as
$t\to\infty$\,.
It is endowed with a uniform norm  topology so that  set
$B\subset \bbC_0(\R_+,\R)$ is compact if it is compact as a subset of
$\bbC(\R_+,\R)$ and $\phi(t)\to0$\,, as $t\to\infty$\,, uniformly on
$B$\,.
Sequence
$X_n$ of stochastic processes with trajectories in
$\D(\R_+,\R)$ is said to be $\bbC_0$--exponentially tight
of order $r_n$ if it is  
$\bbC$--exponentially tight of order $r_n$ and its all 
 LD limit points are idempotent processses with trajectories 
in $\bbC_0(\R_+,\R)$\,.
\begin{lemma}
   \label{the:c0}
Sequence $X_n$ of real--valued processes
 is $\bbC_0$--exponentially tight of order $r_n$
if and only if it is
$\bbC$--exponentially tight of order $r_n$ and, for all $\epsilon>0$\,,
\begin{equation}
  \label{eq:179}
  \lim_{t\to\infty}\limsup_{n\to\infty}
      P(\abs{X_n(t)}>\epsilon)^{1/r_n}=0\,.
\end{equation}
 \end{lemma}
 \begin{proof}
 It needs to be proved that
\eqref{eq:179} implies that the LD limit points 
$\Pi$ of $X_n$ are supported by $\bbC_0(\R_+,\R)$\,, i.e.,
$\Pi(\bbC(\R_+,\R)\setminus\bbC_0(\R_+,\R))=0$\,.
By the definition of LD convergence,
\[
  \liminf_{n\to\infty}
      P(\abs{X_n(t)}>\epsilon)^{1/r_n}\ge 
\Pi(\abs{X(t)}>\epsilon)\,.
\]
Hence,
\[
  \lim_{t\to\infty}\Pi(\abs{X(t)}>\epsilon)=0\,.
\]
The latter is equivalent to the convergence
\begin{equation}
  \label{eq:194}
\lim_{t\to\infty}  \sup_{X:\, \Pi(X)\ge a}\abs{X(t)}=0 \,, \text{ for any }
a>0\,.
\end{equation}
(As a byproduct, the set $\{X:\, \Pi(X)\ge a\}$ is  compact in
$\bbC_0(\R_+,\R)$\,.)
 \end{proof}

\section{ LD convergence of mixtures}\label{cont_conv}
The main result of this section is Lemma 
\ref{le:LDconv_mixt} on the LD convergence of mixtures of
probability distributions. The next lemma is a special case of
Lemma A.1 in Puhalskii \cite{Puh23a}.
\begin{lemma}
  \label{le:cont}
Let $S$ be a metric space and $P_n$ be a sequence of probability
measures
on the Borel $\sigma$-algebra of $S$\,. Suppose that the sequence $P_n$
LD converges at rate $n$ to deviability $\Pi$ on $S$\,.
 Let $h_n$  be 
$\R_+$--valued bounded Borel functions on $S$ and let $h$ be an
$\R_+$--valued function on $S$\,.
If 
$h_n(y_n)\to h(y)$\,, as $n\to\infty$\,, for every sequence $y_n$ of
elements of $S$ and every $y$ such that $y_n\to y$\,,
 $y_n$ belongs to the support of $P_n$\,, and $\Pi(y)>0$\,, then
 \[
   \lim_{n\to\infty}\bl(\int_Sh_n(y)^n\,P_n(dy)\br)^{1/n}=
\sup_{y\in S}h(y)\Pi(y)\,.
 \]
\end{lemma}
\begin{lemma}
  \label{le:LDconv_mixt}
Let $S$ and $S'$ be metric spaces endowed with Borel $\sigma$--algebras. Let, for each $n$\,,
 $P_{n,y}(dy')$ be a Markov kernel from $S$ to $S'$ and $\mu_n(dy)$ be a
 probability measure on the Borel $\sigma$--algebra of $S$\,. Let
 $\Pi_y$ be a deviability on $S'$\,, for each $y\in S$\,, and let
 $\mu$ be a deviability on $S$\,. Let $
 P_n(\Gamma')=\int_SP_{n,y}(\Gamma')\,\mu_n(dy)$\,, for Borel
 $\Gamma'\subset S'$\,, and let $\Pi(\Gamma')=\sup_{y\in
   S}\Pi_y(\Gamma')\mu(y)$\,, for $\Gamma'\subset S'$\,.
Suppose that the deviabilites $\Pi_y$ are uniformly tight over $y$
from compact sets, i.e., for arbitrary $\epsilon>0$ and compact
$K\subset S$\,,
 there exists compact $K'\subset S'$ such that
$\sup_{y\in K}\Pi_y(S'\setminus K')\le \epsilon$\,, and 
   that the function $\Pi_y(y')$ is upper
semicontinuous in $(y,y')\in S\times S'$\,.
Suppose that the sequence $\mu_n$ LD converges to $\mu$ at rate $n$ and 
the sequence $P_{n,y_n}$ of probability measures on $S'$ LD converges to $\Pi_y$\,, for $\mu$--almost every
$y\in S$ and every sequence $y_n\to y$ such that $y_n$ belongs to the
support of $\mu_n$\,. Then, $\Pi$ is a deviability on $S'$ and
the sequence $P_n$ LD converges to $\Pi$ at rate $n$\,. \end{lemma}
\begin{proof}
In order to verify that $\Pi$ is a deviability, one needs to check
that the function $\sup_{y\in S}\Pi_y(y')\mu(y)$ is upper compact.
Given $\gamma>0$\,, owing to $\mu$ being a deviability, the set 
$K=\{y\in S:\, \mu(y)\ge \gamma\}$ is compact and
 $\mathcal{T}=\{y':\,\sup_{y\in S}\Pi_y(y')\mu(y)\ge\gamma\}
=\{y':\,\sup_{y\in K
  }\Pi_y(y')\mu(y)\ge\gamma\}$\,. 
To see that $\mathcal{T}$ is compact, let $y'_n\in\mathcal{T}$\,. 
Since $\sup_{y\in K}\Pi_y(y'_n)\ge\gamma$ and 
the deviabilities $\Pi_y$ are tight uniformly over compact sets of
$y$\,,
  the $y'_n$ form a relatively
compact set. 
Let $y_n\in K$ deliver the sups for $y'_n$ in the definition of
$\mathcal{T}$\,. 
Suppose that $(y_n,y'_n)\to(y,y')$\,. Since 
$\Pi_{y_n}(y'_n)\mu(y_n)\ge\gamma$\,,  by upper semicontinuity,
$\Pi_{y}(y')\mu(y)\ge\gamma$\,, as needed.

In order to verify the LD convergence, let $f$ represent an
  $\R_+$--valued bounded continuous function on $S'$\,.
Let, for $y\in S$\,,
\[
  h_n(y)=\bl(\int_{S'}f(y')^n P_{n,y}(dy')\br)^{1/n}
\]
and 
\[
  h(y)=\sup_{y'\in S'}f(y')\Pi_{y}(y')\,.
\]
By hypotheses, if, given $y_n\in S$ and $y\in S$\,,
$y_n\to y$\,, $y_n$ belongs to the support of $\mu_n$ and $\mu(y)>0$\,,
then $h_n(y_n)\to h(y)$\,. Therefore, 
by Lemma \ref{le:cont},
\[
  \bl(\int_{S'}f(y')^nP_n(dy')\br)^{1/n}=\bl(\int_S
  h_{n'}(y)^{n'}\,\mu_{n'}(dy)\br)^{1/n'}
\to
\sup_{y\in S}
h(y)\mu(y)
=\sup_{y'\in S'}f(y')\Pi(y')\,.
\]
\end{proof}
\begin{remark}
  The lemma is essentially  subsumed by
 Theorem 2 in Biggins \cite{Big04}, a notable distinction being that 
condition {\bf exp-cty} in Biggins \cite{Big04} is required  to be true
for all $\tilde\theta\in\tilde \Theta$\,. This lemma shows that the requirement can be
weakened to apply only to  
$\tilde\theta$ with  $\psi(\tilde\theta)<\infty$ (in the notation of 
Biggins \cite{Big04}), which
corresponds to the condition that $\mu(y)>0$ in Lemma \ref{le:LDconv_mixt}.
Besides, the presentation here follows a
 different path and is more compact.
\end{remark}


\def\cprime{$'$} \def\cprime{$'$} \def\cprime{$'$} \def\cprime{$'$}
  \def\cprime{$'$} \def\cprime{$'$} \def\cprime{$'$} \def\cprime{$'$}
  \def\cprime{$'$} \def\cprime{$'$} \def\cprime{$'$}
  \def\polhk#1{\setbox0=\hbox{#1}{\ooalign{\hidewidth
  \lower1.5ex\hbox{`}\hidewidth\crcr\unhbox0}}} \def\cprime{$'$}
  \def\cprime{$'$} \def\cprime{$'$}

\end{document}